\documentclass{amsart}

\usepackage[utf8]{inputenc}
\usepackage[T1]{fontenc}
\usepackage{enumerate}
\usepackage{xcolor,pgf,tikz,pgflibraryarrows,pgffor,pgflibrarysnakes}
\usepackage{float}
\usepackage{caption}
\usepackage{subcaption}

\numberwithin{equation}{section}
\newtheorem{theorem}{Theorem}[section]

\newtheorem{cor}[theorem]{Corollary}
\newtheorem{prop}[theorem]{Proposition}

\newtheorem{question}{Question}

\theoremstyle{definition}
\newtheorem{definition}[theorem]{Definition}

\theoremstyle{remark}
\newtheorem{remark}[theorem]{Remark}

\newcommand{\A}{\mathcal{A}}

\author[Q. Menet]{Quentin Menet}
\address[Q. Menet]{Service de Probabilit\'e et Statistique, D\'epartement de Math\'ematique\\ Universit\'{e} de Mons\\ Place du Parc 20\\ 7000 Mons, Belgium}
\email{quentin.menet@umons.ac.be}

\author[D. Papathanasiou]{Dimitris Papathanasiou}
\address[D. Papathanasiou]{Sabanci University Tuzla Campus, Orta Mahalle, \"{U}niversite Cadesi No:27 Tuzla, 34956 Istanbul, Turkey}
\email{d.papathanasiou@sabanciuniv.edu}

\thanks{The first author is a Research Associate of the Fonds de la Recherche Scientifique - FNRS}

\subjclass[2020]{47A16}

\keywords{Weighted shift, frequent hypercyclicity, chaos}

\begin{document}

\title[Dynamics of weighted shifts on $\ell^p$-sums and $c_0$-sums]{Dynamics of weighted shifts on $\ell^p$-sums and $c_0$-sums}

\maketitle

\begin{abstract}
We investigate a generalization of weighted shifts where each weight $w_k$ is replaced by an operator $T_k$ going from a Banach space $X_k$ to another one $X_{k-1}$. We then look if the obtained shift operator $B_{(T_k)}$ defined on the $\ell^p$-sum (or the $c_0$-sum) of the spaces $X_k$ is hypercyclic, weakly mixing, mixing, chaotic or frequently hypercyclic. We also compare the dynamical properties of $T$ and of the corresponding shift operator $B_T$. Finally, we interpret some classical criteria in Linear Dynamics in terms of the dynamical properties of a shift operator.
\end{abstract}

\section{Introduction}

An important family of operators in Linear Dynamics is given by the family of weighted shifts on $\ell^p$ ($1\le p<\infty$) or on $c_0$. These operators can help to get some interesting examples and counterexamples but can also help to better understand some dynamical properties through their characterizations for the weighted shifts. 

In this paper, we will focus on the five most important notions in linear dynamics: hypercyclicity, weak mixing, mixing, chaos (in the sense of Devaney) and frequent hypercyclicity. All these properties are completely characterized in terms of weights for the weighted shifts on $\ell^p$ or on $c_0$ (see \cite{BaRu15}, \cite{CoSa}, \cite{Gro00}, \cite{Salas}).

These five notions are defined as follows:

\begin{definition}
Let $X$ be a Banach space and $T\in L(X)$.
\begin{enumerate}
\item $T$ is hypercyclic if there exists $x\in X$ such that $\text{Orb}(x,T):=\{T^n x: n\ge 0\}$ is dense in $X$.
\item $T$ is weakly mixing if $T\oplus T$ is hypercyclic on $X\oplus X$.
\item $T$ is mixing if for every non-empty open sets $U$, $V$ in $X$, the set $N_T(U,V):=\{n\ge 0:T^n U\cap V\ne \emptyset\}$ is cofinite.
\item $T$ is chaotic (in the sense of Devaney) if $T$ is hypercyclic and $T$ possesses a dense set of periodic points.
\item $T$ is  frequently hypercyclic if there exists $x\in X$ such that for every non-empty open set $U$ in $X$, the set $N_T(x,U):=\{n\ge 0:T^nx\in U\}$ has a positive lower density.
\end{enumerate}
\end{definition}

More information on these notions can be found in the two books \cite{BM09, GrPe11}. In general, we have the following implications between these five notions and no other implication is true (see \cite[Section 4]{Me1} for more details). 
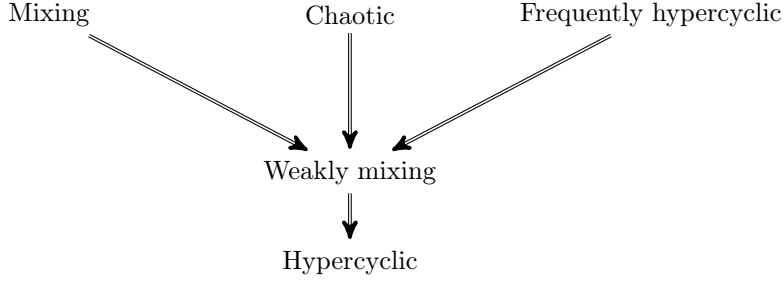
\begin{figure}[H]
    \begin{center}
      \begin{tikzpicture}[xscale=1,yscale=1.5,>=stealth',shorten >=1pt]
%        \everymath{\scriptstyle}
       
        \path (6,4) node[] (q0) {Frequently hypercyclic};
        \path (2,4) node[] (q1) {Chaotic};
        \path (-2,4) node[] (q2) {Mixing};

        \path (2,2.6) node[] (q3) {Weakly mixing};
        \path (2,1.8) node[] (q4) {Hypercyclic};

        \draw[double,arrows=->] (q0) -- (q3);
        \draw[double,arrows=->] (q1) -- (q3);
        \draw[double,arrows=->] (q2) -- (q3);
        \draw[double,arrows=->] (q3) -- (q4);

        \end{tikzpicture}
    \end{center}
    \caption{Links between different notions in Linear Dynamics}
\label{fig}
\end{figure}

However, if we restrict ourselves to the study of weighted shifts on $\ell^p$ or on $c_0$, more implications are true. For instance, a hypercyclic weighted shift is always weakly mixing and a chaotic weighted shift is always frequently hypercyclic and mixing. In fact, on $\ell^p$, a weighted shift $B_w$ is chaotic if and only if it is frequently hypercyclic, and on $c_0$, a weighted shift $B_w$ is chaotic if and only if it is mixing. These links are depicted in the following figures and we know that no other implication is true (see \cite{BaGr07}, \cite{BaRu15}, \cite{CoSa}, \cite{Gro00}, \cite{Salas}).
\begin{figure}[H]
\subcaptionbox{Links for weighted shifts on $\ell^p$}[.4\textwidth]{
      \begin{tikzpicture}[xscale=1,yscale=1.5,>=stealth',shorten >=1pt]
%        \everymath{\scriptstyle}
       
        \path (0.8,4.2) node[] (q0) {Frequently hypercyclic};
        \path (-4,4.2) node[] (q1) {Chaotic};
        \path (-2,3.4) node[] (q2) {Mixing};

        \path (-2,2.6) node[] (q3) {Weakly mixing};
        \path (-2,1.7) node[] (q4) {Hypercyclic};

		\draw[double,arrows=<->] (q0) -- (q1);
        \draw[double,arrows=->] (q1) -- (q2);
        \draw[double,arrows=->] (q0) -- (q2);
        \draw[double,arrows=->] (q2) -- (q3);
        \draw[double,arrows=<->] (q3) -- (q4);
        \end{tikzpicture}
}
\hspace*{2.2cm}
\subcaptionbox{Links for weighted shifts on $c_0$}[.4\textwidth]{
      \begin{tikzpicture}[xscale=1,yscale=1.5,>=stealth',shorten >=1pt]
%        \everymath{\scriptstyle}
       
        \path (6,4.2) node[] (q0) {Mixing};
        \path (2,4.2) node[] (q1) {Chaotic};
        \path (4,3.4) node[] (q2) {Frequently hypercyclic};

        \path (4,2.6) node[] (q3) {Weakly mixing};
        \path (4,1.7) node[] (q4) {Hypercyclic};

		\draw[double,arrows=<->] (q0) -- (q1);
        \draw[double,arrows=->] (q1) -- (q2);
        \draw[double,arrows=->] (q0) -- (q2);
        \draw[double,arrows=->] (q2) -- (q3);
        \draw[double,arrows=<->] (q3) -- (q4);

        \end{tikzpicture}
}
\end{figure}
In view of the importance of weighted shifts in linear dynamics, several generalizations of these operators have already been introduced and investigated such as weighted pseudoshifts \cite{Gro00} or more recently weighted shifts on trees (see \cite{GrPa21}, \cite{JJS12}, \cite{Mar17}). In this paper we consider a version of weighted shifts defined on an infinite product of Banach spaces where the weights are replaced by operators. \\

To be able to deal simultaneously with unilateral and bilateral shifts, we will denote by $J$ the set $\mathbb{N}$ (resp. $\mathbb{Z}$) and by $J^*$ the set $\mathbb{N}\backslash\{1\}$ (resp. $\mathbb{Z}$). 

\begin{definition}\label{defshift}
Let $(X_k)_{k\in J}$ be a sequence of sets.
Given a sequence of maps $T_k:X_k\to X_{k-1}$ with $k\in J^*$, we define $B_{(T_k)_{k\in J^*}}$ on $\prod_{k\in J} X_k$ by 
\[B_{(T_k)_{k\in J^*}} (x_i)_{i\in J}=(T_{k+1}x_{k+1})_{k\in J}.\]
If we have $T_k=T$ and $X_k=X$ for all $k\in J$, we will use the notation $B_T$ for the map $B_{(T_k)_{k\in J^*}}$.
\end{definition}

In this paper, we will focus on the case where each $T_k$ is an operator, each space $X_k$ is a Banach space and the map $B_{(T_k)_{k\in J^*}}$ gives us an operator on $\ell^p((X_k)_k,J)$ or on $c_0((X_k)_k,J)$. We recall that
\[\ell^p((X_k)_k,J)=\{(x_k)_{k\in J}\in \prod_{k\in J} X_k: \sum_{k\in J} \|x_k\|^p<\infty\}\]
and
\[c_0((X_k)_k,J)=\{(x_k)_{k\in J}\in \prod_{k\in J} X_k: \lim_{|k|\to \infty} \|x_k\|=0\}.\]
Moreover, if each $X_k$ is equal to $X$, we will use the notations $\ell^p(X,J)$ and $c_0(X,J)$.\\

Notice that we get back the usual weighted shifts by considering $X_k=\mathbb{K}$ and $T_k=w_k \text{Id}$ where $\text{Id}$ is the identity operator. We can also get the weighted shifts on trees by considering for $X_k$ the $\ell^p$-space (or the $c_0$-space) supported by the $k$-th generation. This idea to replace weights by operators has been considered simultaneously and independently by Carvalho, Darji and Varandas \cite{CaDaVa}. In this one, they consider the additional assumption that each operator $T_k$ is an invertible operator on some Banach space $X$ and investigate other dynamical properties : the shadowing property and the generalized hypercbolicity.\\

The main goal of this paper consists in providing some answers to the following natural questions : 
\begin{question}
How can we characterize that an operator $B_{(T_k)}$ is hypercyclic, weakly mixing, mixing, chaotic or frequently hypercyclic in terms of the sequence $(T_k)$?
\end{question}
\begin{question}
Do the links depicted in Figures (A) and (B) for the weighted shifts $B_w$ can be extended to the operators $B_{(T_k)}$ on $\ell^p((X_k)_k,J)$ and $c_0((X_k)_k,J)$?
\end{question}
\begin{question}
Which dynamical properties an operator $T$ can inherit from $B_T$? Which dynamical properties the operator $B_T$ can inherit from $T$? Is there a quasi-conjugacy between $T$ and $B_T$? 
\end{question}

The answers at these questions are a mix of expected results and surprises.\\

For instance, in section~\ref{hyp}, we show as expected that $B_{(T_k)_{k\in J^*}}$ is hypercyclic on $\ell^p((X_k)_k,J)$ or on $c_0((X_k)_k,J)$ if and only if $B_{(T_k)_{k\in J^*}}$ is weakly mixing, like in Figures (A) and (B). However there are important differences between the unilateral case ($J=\mathbb{N}$) and the bilateral case ($J=\mathbb{Z}$). In the unilateral case, we show that if $T$ is hypercyclic on $X$ then $B_T$ is hypercyclic on $\ell^p(X,\mathbb{N})$ and on $c_0(X,\mathbb{N})$. Notice that $B_T$ can be hypercyclic on $\ell^p(X,\mathbb{N})$ or on $c_0(X,\mathbb{N})$ even if $T$ is not hypercyclic on $X$; it suffices to consider $T=2 \text{Id}$ and $X=\mathbb{K}$. On the other hand, in the bilateral case, we have the nice following equivalence : $B_T$ is hypercyclic on $\ell^p(X,\mathbb{Z})$ or $c_0(X,\mathbb{Z})$ if and only if $T$ is weakly mixing on $X$. However, since there exist hypercyclic operators that are not weakly mixing (see \cite{RoRe}), we deduce that in the bilateral case, it is possible that $T$ is hypercyclic on $X$ but $B_T$ is not hypercyclic on $\ell^p(X,\mathbb{Z})$ or on $c_0(X,\mathbb{Z})$!

In section~\ref{MCF}, we characterize when $B_{(T_k)}$ is mixing and we deduce that in the bilateral case, $T$ is mixing if and only if $B_T$ is mixing on $\ell^p(X,\mathbb{Z})$ or on $c_0(X,\mathbb{Z})$. In view of the previous results, one can wonder if in the bilateral case, $T$ is quasi-conjugate to $B_T$. By investigating the chaos, we will remark that the answer \emph{depends on $p$}. Indeed, $T$ is not quasi-conjugate to $B_T$ on $\ell^p(X,\mathbb{Z})$ for $p>1$ or on $c_0(X,\mathbb{Z})$  but $T$ is always quasi-conjugate to $B_T$ on $\ell^1(X,\mathbb{Z})$. We also show that if $B_{(T_k)_{k\in J^*}}$ is chaotic on $\ell^p((X_k)_k,J)$ (resp. $c_0((X_k)_k,J)$) then $B_{(T_k)_{k\in J^*}}$ is mixing and frequently hypercyclic. However, unlike the case of weighted shifts, there exists a mixing operator $B_T$ on $c_0(X,\mathbb{N})$ such that $B_T$ is not chaotic, and there exists a frequently hypercyclic operator $B_T$ on $\ell^p(X,\mathbb{N})$ such that $B_T$ is not chaotic. We can therefore show that the relations between the five investigated dynamical properties for the family of operators $(B_{(T_k)})$ are as depicted below and no other implication is true. 

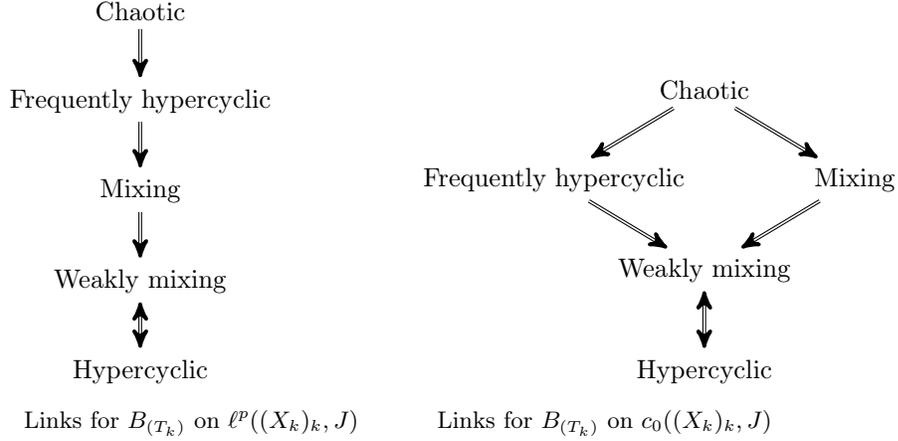
\begin{figure}[H]

\begin{subfigure}{.4\textwidth}
      \begin{tikzpicture}[xscale=1,yscale=1.5,>=stealth',shorten >=1pt]
%        \everymath{\scriptstyle}
       
        \path (-2,5) node[] (q1) {Chaotic};
        \path (-2,4.2) node[] (q0) {Frequently hypercyclic};
        \path (-2,3.4) node[] (q2) {Mixing};

        \path (-2,2.6) node[] (q3) {Weakly mixing};
        \path (-2,1.8) node[] (q4) {Hypercyclic};

		\draw[double,arrows=->] (q1) -- (q0);
        \draw[double,arrows=->] (q0) -- (q2);
        \draw[double,arrows=->] (q2) -- (q3);
        \draw[double,arrows=<->] (q3) -- (q4);

        \end{tikzpicture}
    \caption*{Links for $B_{(T_k)}$ on $\ell^p((X_k)_k,J)$}
\label{fig2}
\end{subfigure}
\hspace*{0.2cm}
\begin{subfigure}{.4\textwidth}
      \begin{tikzpicture}[xscale=1,yscale=1.5,>=stealth',shorten >=1pt]
%        \everymath{\scriptstyle}
       
        \path (6,3.4) node[] (q2) {Mixing};
        \path (4,4.2) node[] (q1) {Chaotic};
        \path (2,3.4) node[] (q0) {Frequently hypercyclic};

        \path (4,2.6) node[] (q3) {Weakly mixing};
        \path (4,1.7) node[] (q4) {Hypercyclic};

		\draw[double,arrows=->] (q1) -- (q0);
        \draw[double,arrows=->] (q1) -- (q2);
         \draw[double,arrows=->] (q0) -- (q3);        
         \draw[double,arrows=->] (q2) -- (q3);
        \draw[double,arrows=<->] (q3) -- (q4);

        \end{tikzpicture}
    \caption*{Links for $B_{(T_k)}$ on $c_0((X_k)_k,J)$}
\label{fig3}
\end{subfigure}
\end{figure}

In section~\ref{Criteria}, we interpret some classical hypercyclicity criteria on $T$ in terms of $B_T$. More precisely, we remark that some version of the Kitai Criterion is equivalent to require that $B_T$ is chaotic on $c_0(X,\mathbb{Z})$ and some version of the Frequent Hypercyclicity Criterion is equivalent to require that $B_T$ is chaotic on $\ell^1(X,\mathbb{Z})$. We then investigate in depth the link between different versions of these criteria and the dynamical properties of $B_T$.

\section{Hypercyclicity and weakly mixing}\label{hyp}
We recall that we will denote by $J$ the set $\mathbb{N}$ (resp. $\mathbb{Z}$) and by $J^*$ the set $\mathbb{N}\backslash\{1\}$ (resp. $\mathbb{Z}$). Let $(X_k,\|\cdot\|_k)_{k\in J}$ be Banach spaces and $T_k\in L(X_k, X_{k-1})$, $k\in J^*$.  We will also let 
\[T_{n,n}=I\in L(X_n,X_n)  \quad \text{and} \quad T_{k,n}=T_{k+1}\circ \cdots \circ T_{n}\in L(X_n,X_k) \quad \text{for $k< n$}.\]
We will finally denote by $B_k(x,r)$ the open ball in $X_k$ centered at $x$ with radius $r$ and by $B(x,r)$ the open ball in the considered product space $\ell^p((X_k)_k,J)$ or $c_0((X_k)_{k},J)$. The shift $B_{(T_k)}$ (as defined in Definition~\ref{defshift}) will give us an operator on $\ell^p((X_k)_k,J)$ or $c_0((X_k)_{k},J)$ as soon as $\sup_{k\in J^*}\|T_k\|<\infty$:

\begin{prop}
 The following assertions are equivalent:
\begin{enumerate}[\upshape (1)]
\item $B_{(T_k)}$ maps $\ell^p((X_k)_{k},J)$ (resp. $c_0((X_k)_{k},J)$) in itself;
\item ${\sup_{k\in J^*}\|T_k\|<\infty}$;
\item $B_{(T_k)}$ is a continuous operator on $\ell^p((X_k)_{k},J)$ (resp. $c_0((X_k)_{k},J)$). 
\end{enumerate}
Moreover, we have
$\|B_{(T_k)}\|=\sup_{k\in J^*}\|T_k\|$.
\end{prop}
\begin{proof}
We perform the proof for $\ell^p((X_k)_{k},J)$ (the case of $c_0((X_k)_{k},J)$ being similar). 

If $\sup_{k\in J^*}\|T_k\|=K<\infty$, then for any $x\in \ell^p((X_k)_{k},J)$, we have
\[\|B_{(T_k)}x\|^p=\sum_{k\in J^*}\|T_kx_k\|_{k-1}^p\le \sum_{k\in J^*}\|T_k\|^p \|x_k\|_k^p\le K^p \|x\|^p.\]
The map $B_{(T_k)}$ is thus continuous and $\|B_{(T_k)}\|\le\sup_{k\in J^*}\|T_k\|$.

On the other hand, if $B_{(T_k)}$ is a continuous operator on $\ell^p((X_k)_{k},J)$ then for any $k\in J^*$, we have
\[\|T_k\|=\sup_{\substack{x_k\in X_k,\\ \|x_k\|_k\le 1}}\|T_kx_k\|_{k-1}\le \sup_{\substack{x\in \ell^p((X_k)_{k},J),\\ \|x\|\le 1}}\|B_{(T_k)}x\|=\|B_{(T_k)}\|,\]
since $\|(0,\dots,0,x_k,0,\dots)\|=\|x_k\|_k$.
The assertion $(2)$ is thus equivalent to $(3)$ and
\[\|B_{(T_k)}\|=\sup_{k\ge 1}\|T_k\|.\]

We also remark that each coordinate map $P_n:\ell^p((X_k)_{k},J)\to X_n$ defined by $P_n((x_k)_{k\in J})=x_n$ is continuous. Therefore, if $B_{(T_k)}$ maps $\ell^p((X_k)_{k},J)$ in itself, we deduce that 
$B_{(T_k)}$ is continuous by using the closed graph theorem. Finally, since it is obvious that $(3)$ implies $(1)$, we get the desired result.
\end{proof}

It is well-known that if a hypercyclic operator $T$ possesses a dense set of vectors whose orbits tend to $0$ then $T$ is weakly mixing (see \cite{Gri}). We can therefore deduce that in the unilateral case, every hypercyclic operator $B_{(T_k)}$ on $\ell^p((X_k)_{k},\mathbb{N})$ or $c_0((X_k)_{k},\mathbb{N})$ is actually weakly mixing. To this end, we first investigate in terms of $(T_k)_{k\in \mathbb{N}\backslash\{1\}}$ when $B_{(T_k)}$ is hypercyclic (or equivalently weakly mixing).

\begin{prop}\label{unilhyp}
Let $B_{(T_k)}$ be an operator on $\ell^p((X_k)_{k},\mathbb{N})$ (resp. $c_0((X_k)_{k},\mathbb{N})$).
The following assertions are equivalent:
\begin{enumerate}[\upshape (1)]
\item $B_{(T_k)}$ is hypercyclic;
\item $B_{(T_k)}$ is weakly mixing;
%\item $\bigcup_{n\ge 1} B_{(T_k)}^n(B(0,1))$ is dense in %$l^p((X_k)_{k\ge 0})$ (resp. $c_0(X_k)_{k\geq 0}$);
\item For every $N\geq 1$, for every $(U_j)_{1\leq j\leq N}$, where each set $U_j$ is a non-empty open subsets of $X_j$, there exists $n\ge 1$ such that for every $1\leq j\leq N$
$$
T_{j,j+n}B_{j+n}(0,1)\cap U_j \neq \emptyset.
$$
\end{enumerate}
\end{prop}
\begin{proof}
We already know that $B_{(T_k)}$ is weakly mixing if and only if $B_{(T_k)}$ is hypercyclic.\\
Assume now that $B_{(T_k)}$ is hypercyclic. Let $N\ge 1$, and $(U_{j})_{1\le j\le N}$ where each set $U_j$ is a non-empty open subset of $X_j$. We can consider $\varepsilon>0$ and $x$ such that for every $1\le j\le N$,  $B_j(x_j,\varepsilon)\subset U_j$ and $x_j=0$, for $j>N$. Let $y$ be a hypercyclic vector for $B_{(T_k)}$ with $\|y\|<1$. We pick $n\geq 1$ such that 
\[\|B^n_{(T_k)}y-x\|<\varepsilon.\]
If $1\le j\le N$, we deduce that $y_{j+n}\in B_{j+n}(0,1)$ and that  
\[T_{j,j+n}y_{j+n}\in B_{j}(x_j,\varepsilon)\subset U_j.\]
In other words, the assertion (3) is satisfied.\\

On the other hand, assume that for every $N\ge 1$ and every $(U_{j})_{1\le j\le N}$ where $U_j$ is a non-empty open subset in $X_j$, there exists $n\ge 1$ such that for every $1\le j\le N$,
\[T_{j,j+n}B_{j+n}(0,1)\cap U_j\ne \emptyset.\]
 Since $B_{(T_k)_k}$ has a dense generalized kernel, we can deduce that $B_{(T_k)_k}$ is topologically transitive and thus hypercyclic (and even weakly mixing), if we show that for every $\varepsilon>0$, every $x\in  c_{00}((X_k)_k,\mathbb{N})$ (the space of finitely supported sequences), there exists $n\ge 1$ such that
\[B^n_{(T_k)}(B(0,\varepsilon))\cap B(x,\varepsilon)\ne \emptyset .\]
Let $\varepsilon>0$ and $x\in c_{00}((X_k)_k,\mathbb{N})$. We consider $N\ge 1$, such that if $j>N$ then $x_j=0$, and the open sets $(U_j)_{1\le j\le N}$ given by $U_j=B_j(\frac{N}{\varepsilon}x_j,1)$. We deduce from our assumption that there exists $n\geq 1$ such that 
for every $1\le j\le N$,
\[T_{j,j+n}B_{j+n}(0,1)\cap U_j\ne \emptyset.\]
For every $1\le j\le N$, there thus exists $x'_{j+n}\in X_{j+n}$ such that $\|x'_{j+n}\|_{j+n}<1$ and $T_{j,j+n}x'_{j+n}\in U_j$. If we complete by letting $x'_j=0$ for $j\le n$ or $j>n+N$, we deduce
that  $\frac{\varepsilon}{N} x'\in B(0,\varepsilon)$ and \[B^n_{(T_k)}(\frac{\varepsilon}{N} x')\in B(x,\varepsilon).\]
We can thus conclude the proof.
\end{proof}

In the case of an operator $B_T$, it is natural to investigate the link between the dynamical properties of $T$ and $B_T$. The third condition in Proposition~\ref{unilhyp} can be simplified as follows in this case.

\begin{cor}\label{corhyp}
Let $B_{T}$ be an operator on $\ell^p(X,\mathbb{N})$ or $c_0(X,\mathbb{N})$.
The following assertions are equivalent:
\begin{enumerate}[\upshape (1)]
\item $B_{T}$ is hypercyclic;
\item $B_{T}$ is weakly mixing;
%\item for any $N\ge 0$, \[\bigcup_{n\ge 1} (T\times\cdots\times T)^n(B(0,1)\times \cdots \times B(0,1))\] is dense in $X^N$;
\item for any $N\ge 0$, any family $(U_j)_{1\le j\le N}$ of non-empty open sets in $X$,
\[N_{T}(B_X(0,1),U_1)\cap\cdots\cap N_{T}(B_X(0,1),U_N)\ne \emptyset.\]
\end{enumerate}
\end{cor}

In particular, if $T$ is weakly mixing, then $B_T$ is hypercyclic. Since there exist hypercyclic operators which are not weakly mixing, we can wonder if under the assumption that $T$ is hypercyclic then $B_T$ is hypercyclic (and hence weakly mixing). 

\begin{prop}
    If $T$ is hypercyclic on $X$, then $B_T$ is hypercyclic on $\ell^p(X,\mathbb{N})$ and $c_0(X,\mathbb{N})$.
\end{prop}

\begin{proof}
    Let $N\geq 1$ and a family $(U_j)_{1\le j\le N}$ of non-empty open sets in $X$. By topological transitivity of $T$, setting $U_N'=U_N$, we may find, for each $2\leq j\leq N$, an integer $m_j\in \mathbb{N}$ and a non-empty open subset $U_{j-1}'\subset U_{j-1}$ such that $T^{m_j}U_{j-1}'\subset U_j'$. By continuity of $T$, we may also choose a neighbourhood $W$ of $0$ such that $T^iW\subset B_X(0,1)$, for each $0\leq i\leq m_2+\dots +m_N$. Again using topological transitivity, we pick $m_1\in \mathbb{N}$ such that $T^{m_1}W\cap U_1'\neq \emptyset$. Therefore, there is $x\in W$ such that $T^{m_1}x\in U_1'$. This implies that $T^{m_1+\dots +m_i}x\in U_i'$ for each $1\leq i\leq N$. Since, for $i\geq 2$,  $T^{m_2+\dots +m_i}x\in B(0,1)$, we conclude that \[m_1\in N_T(B_X(0,1),U_1)\cap \dots \cap N_T(B_X(0,1),U_N)\] which gives that $B_T$ is hypercyclic by Corollary~ \ref{corhyp}.
\end{proof}

Notice that $B_T$ can be hypercyclic on $\ell^p(X,\mathbb{N})$ or $c_0(X,\mathbb{N})$ even if $T$ is not hypercyclic.
A simple example is given by $T=2Id$ on $\mathbb{K}$. In particular, there exists some operator $T$ such that \textbf{$T$ is not quasi-conjugate to $B_T$ on $\ell^p(X,\mathbb{N})$ or $c_0(X,\mathbb{N})$}. However, in some cases, the operator $T$ can be quasi-conjugate to $B_T$ on $\ell^p(X,\mathbb{N})$ or $c_0(X,\mathbb{N})$. This is the subject of the following result.

\begin{prop}\label{comparison}
    Let $T=B_{(T_k)}$ be an operator on $X=\ell^p((X_k)_k,\mathbb{N})$ (resp. $c_0((X_k)_k,\mathbb{N}))$). The operator $T$ is quasi-conjugate to $B_T$ on $\ell^{p}(X,\mathbb{N})$ (resp. $c_0(X,\mathbb{N)}$).
\end{prop}
\begin{proof}
Let $Y=\ell^{p}(X,\mathbb{N})$ or $c_0(X,\mathbb{N})$.
It suffices to consider the continuous map $\phi:Y\to X$ given by
    $$
    \phi:((x_{1,1},x_{1,2},\dots),(x_{2,1},x_{2,2},\dots),\dots) \mapsto (x_{1,1},x_{2,2},\dots).
    $$
It is not difficult to check that $\phi$ is a linear surjective contraction, which satisfies that $\phi\circ B_{T}=T\circ \phi$.
\end{proof}
\begin{remark}\label{comprem}
Notice that for every operator $B_{(T_k)_{k\in \mathbb{Z}}}$ defined on $\ell^p((X_k)_k,\mathbb{Z})$ (resp. $c_0((X_k)_k,\mathbb{Z}))$), the restriction $B_{(T_k)_{k\in \mathbb{N}}}$ on $\ell^p((X_k)_k,\mathbb{N})$ (resp. $c_0((X_k)_k,\mathbb{N}))$) is always quasi-conjugate to $B_{(T_k)_{k\in \mathbb{Z}}}$ on $\ell^p((X_k)_k,\mathbb{Z})$ (resp. $c_0((X_k)_k,\mathbb{Z}))$).
\end{remark}

In particular, if $T$ is a weighted shift on $X=\ell^p(\mathbb{N})$ (resp. $c_0(\mathbb{N})$)  or a weighted shift on trees then $T$ is quasi-conjugate to $B_T$ on $\ell^{p}(X,J)$ (resp. $c_0(X,J)$). We will use several times this fact to establish different counterexamples.

As we will see below, there will be several differences between the unilateral case and the bilateral case. We can already notice that there is no operator $T$ on $\mathbb{K}$ such that $B_T$ is hypercyclic on $\ell^p(\mathbb{K},\mathbb{Z})$ or $c_0(\mathbb{K},\mathbb{Z})$. Let's start by adapting Proposition~\ref{unilhyp} to the bilateral case.

\begin{prop}\label{bilhyp}
On $c_0((X_k)_k,\mathbb{Z})$ and $\ell^p((X_k)_k,\mathbb{Z})$ with $1\le p<\infty$, the following assertions are equivalent:
\begin{enumerate}
\item $B_{(T_k)}$ is hypercyclic;
\item $B_{(T_k)}$ is weakly mixing;
\item for every $N\ge 1$, for every $(U_{j})_{-N\le j\le N}$, $(V_{j})_{-N\le j\le N}$ where the sets $U_j$ and $V_j$ are non-empty open subsets in $X_j$, there exists $n$ such that for every $-N\le j\le N$,
\[T_{j-n,j}U_j\cap B_{j-n}(0,1)\ne \emptyset \quad \text{and} \quad T_{j,j+n}B_{j+n}(0,1)\cap V_j\ne \emptyset.\]
\end{enumerate}
\end{prop}
\begin{proof}
Assume that $B_{(T_k)}$ is hypercyclic. Let $N\ge 1$, $(U_{j})_{-N\le j\le N}$ and $(V_{j})_{-N\le j\le N}$ where $U_j$ and $V_j$ are non-empty open subsets in $X_j$. We can consider $0<\varepsilon<1$ and $x, y$ such that for every $-N\le j\le N$,  $B_j(x_j,\varepsilon)\subset U_j$, $B_j(y_j,\varepsilon)\subset V_j$, and $x_j=y_j=0$ for $|j|>N$. Let $z$ be a hypercyclic vector for $B_{(T_k)}$ with $\|z-x\|<\varepsilon$. We pick $n>2N$ such that 
\[\|B^n_{(T_k)}z-y\|<\varepsilon.\]
If $-N\le j\le N$, we deduce that $z_{j}\in B_j(x_j,\varepsilon)\subset U_j$ and that \[T_{j-n,j}z_{j}\in B_{j-n}(0,\varepsilon)\subset B_{j-n}(0,1).\]
Similarly, we get that if $-N\le j\le N$, $z_{j+n}\in B_{j+n}(0,\varepsilon)\subset B_{j+n}(0,1)$ and that  
\[T_{j,j+n}z_{j+n}\in B_{j}(y_j,\varepsilon)\subset V_j.\]
We have thus shown that (1) $\implies$ (3).\\

On the other hand, assume that for every $N\ge 1$, for every $(U_{j})_{-N\le j\le N}$, $(V_{j})_{-N\le j\le N}$ where $U_j$ and $V_j$ are non-empty open subsets in $X_j$, there exists $n$ such that for every $-N\le j\le N$,
\[T_{j-n,j}U_j\cap B_{j-n}(0,1)\ne \emptyset \quad \text{and} \quad T_{j,j+n}B_{j+n}(0,1)\cap V_j\ne \emptyset.\]
In order to deduce that $B_{(T_k}$ is weakly mixing, it is enough to show that for every $\varepsilon>0$, every $x,y\in c_{00}((X_k)_k,\mathbb{Z})$, there exists $n$ such that
\[B^n_{(T_k)}(B(x,\varepsilon))\cap B(0,\varepsilon)\ne \emptyset \quad \text{and} \quad B^n_{(T_k)}(B(0,\varepsilon))\cap B(y,\varepsilon)\ne \emptyset \quad \text{(see \cite{BeGr}, \cite{Le})}.\]
Let $\varepsilon>0$ and $x,y\in c_{00}((X_k)_k,\mathbb{Z})$. We consider $N$ such that if $|j|>N$ then $x_j=y_j=0$. Let $(U_j)_{-N\le j\le N}$ and $(V_j)_{-N\le j\le N}$ be given by $U_j=B(\frac{(2N+1)}{\varepsilon}x_j,1)$ and $V_j=B(\frac{(2N+1)}{\varepsilon}y_j,1)$. We deduce from our assumption that there exists $n$ such that 
for every $-N\le j\le N$,
\[T_{j-n,j}U_j\cap B_{j-n}(0,1)\ne \emptyset \quad \text{and} \quad T_{j,j+n}B_{j+n}(0,1)\cap V_j\ne \emptyset.\]
For every $-N\le j\le N$, there thus exist $x'_j\in U_j$ such that $\|T_{j-n,n}x'_j\|_{j-n}<1$ and $y'_{j+n}\in X_{j+n}$ such that $\|y'_{j+n}\|_{j+n}<1$ and $T_{j,j+n}y'_{j+n}\in V_j$. If we let $x'_j=0$ if $|j|>N$ and $y'_{j+n}=0$ if $|j|>N$, we deduce
that $\frac{\varepsilon}{2N+1}x'\in B(x,\varepsilon)$ and that \[B^n_{(T_k)_k}(\frac{\varepsilon}{2N+1} x')\in B(0,\varepsilon)\]
and we also deduce that $\frac{\varepsilon}{2N+1} y'\in B(0,\varepsilon)$ and \[B^n_{(T_k)_k}(\frac{\varepsilon}{2N+1} y')\in B(y,\varepsilon).\]
Since every weakly mixing operator is hypercyclic, we get the desired equivalences
\end{proof}

If we consider the operator $B_T$ on $c_0(X,\mathbb{Z})$ and $\ell^p(X,\mathbb{Z})$, we then get the following characterization.

\begin{cor}
On $c_0(X,\mathbb{Z})$ and $\ell^p(X,\mathbb{Z})$ with $1\le p<\infty$, the following assertions are equivalent:
\begin{enumerate}
\item $B_T$ is hypercyclic;
\item $B_T$ is weakly mixing;
\item $T$ is weakly mixing on $X$.
\end{enumerate}
\end{cor}
\begin{proof}
Since $T_k=T$ for any $k$, Condition (3) in Proposition~\ref{bilhyp} can be replaced by:\\
for every $N\ge 1$, every $(U_{j})_{-N\le j\le N}$, $(V_{j})_{-N\le j\le N}$ where $U_j$ and $V_j$ are non-empty open subsets in $X$, there exists $n$ such that for every $-N\le j\le N$,
\[T^nU_j\cap B_X(0,1)\ne \emptyset \quad \text{and} \quad T^n B_X(0,1)\cap V_j\ne \emptyset.\]
By linearity, we can replace $B_X(0,1)$ by any open ball centered at $0$ and thus by any neighbourhood of $0$. Condition (3) is therefore equivalent to require that every direct sum $T\oplus\cdots\oplus T$ is weakly mixing (see \cite{BeGr}, \cite{Le}). Moreover, every direct sum $T\oplus\cdots\oplus T$ is weakly mixing if and only if $T$ is weakly mixing (\cite{Fur}).
\end{proof}

We notice that unlike the unilateral case, there exists a hypercyclic operator $T$ on $X$ such that $B_T$ is not hypercyclic on $c_0(X,\mathbb{Z})$ (resp. $\ell^p(X,\mathbb{Z})$). It suffices to consider for $T$ a hypercyclic operator which is not weakly mixing (\cite{RoRe}). Therefore we can deduce that \textbf{$B_T$ is not quasi-conjugate to $T$ on $\ell^p(X,\mathbb{Z})$ or $c_0(X,\mathbb{Z})$.} On the other hand, while a bilateral weighted shift $B_w$ on $c_0(\mathbb{K},\mathbb{Z})$ or $\ell^p(\mathbb{K},\mathbb{Z})$ can never be hypercyclic if $w$ is constant, the operator $B_{(T_k)}$ can be hypercyclic under the condition that $T_k=T$ for every $k$. However, it is only possible if $T$ is weakly mixing on $X$ and thus if $X$ is infinite-dimensional.

\section{Mixing, Chaos and Frequent hypercyclicity}\label{MCF}

In this section, we investigate three strong dynamical properties that imply hypercyclicity : mixing, chaos and frequent hypercyclicity. This study will allow us to prove that for the operator $B_T$, we do not have in general the equivalence between chaos and frequent hypercyclicity on $\ell^p(X,J)$ and we do not have the equivalence between chaos and mixing on $c_0(X,J)$.

\subsection{Mixing}
We start by characterizing when $B_{(T_k)}$ is mixing in terms of the sequence $(T_k)$.  We give below the proof for the bilateral case. The characterization in the unilateral case can easily be obtained by adapting the following proof. Moreover, in view of the characterization obtained in Proposition~\ref{bilhyp}, the following characterization is not surprising.

\begin{prop} \label{prop mixing}
On $\ell^p((X_k)_k,\mathbb{Z})$ with $1\le p<\infty$ and $c_0((X_k)_k,\mathbb{Z})$, the following assertions are equivalent:
\begin{enumerate}
\item $B_{(T_k)}$ is mixing;
\item for every $j\in \mathbb{Z}$, for every non-empty open subsets $U_j$, $V_j$ of $X_j$, there exists $n_0$ such that for every $n\ge n_0$
\[T_{j-n,j}U_j\cap B_{j-n}(0,1)\ne \emptyset \quad \text{and} \quad T_{j,j+n}B_{j+n}(0,1)\cap V_j\ne \emptyset.\]
\end{enumerate}
\end{prop}

\begin{proof}
We first show that (1) $\Rightarrow$ (2). Let $j\in \mathbb{Z}$ and some non-empty open subsets $U_j$ and $V_j$ of $X_j$. Since the projection $P_j$ onto the $j$-th coordinate is bounded and since $B_{(T_k)}$ is mixing, there exists $n_0$ such that for every $n\geq n_0$, 
\[B^n_{(T_k)}(P_j^{-1}(U_j)) \cap B(0,1)\neq \emptyset\quad\text{and}\quad P_j^{-1}(V_j) \cap B_{(T_k)}^{n}(B(0,1))\neq \emptyset.\] Since $P_l(B(0,1))\subset B_l(0,1)$ for any $l\in \mathbb{Z}$, we get (2).\\

We now show that (2) $\Rightarrow$ (1). The operator $B_{(T_k)}$ is mixing if and only if all the sets $N_{B_{(T_k)}}(U,W)$ and $N_{B_{(T_k)}}(W,V)$ are cofinite when $U$ and $V$ are non-empty open subsets of $\ell^p((X_k)_k,\mathbb{Z})$ (resp. $c_0((X_k)_k,\mathbb{Z})$) and $W$ is a neighbourhood of zero (\cite{GrPe10}). By the homogeneity of the norm, it is enough to show that the sets $N_{B_{(T_k)}}(U,B(0,1))$ and $N_{B_{(T_k)}}(B(0,1),V)$ are cofinite.

Let $U$ and $V$ be non-empty open subsets of $\ell^p((X_k)_k,\mathbb{Z})$ (resp. $c_0((X_k)_k,\mathbb{Z})$). We consider $N\geq 1$, $x\in U$ and $y\in V$ such that $x_j=y_j=0$ for $|j|>N$. We then select for all $-N\leq j\leq N$, $U_j$ and $V_j$ open subsets of $X_j$ such that \[x\in \dots  \times \{0\} \times \prod_{j=-N}^NU_j \times \{0\} \times \dots \subset U\ \text{and}\ y\in \dots \times \{0\} \times \prod_{j=-N}^NV_j \times \{0\} \times  \dots \subset V.\] By $(2)$, there exists $n_0$ such that for each $n\geq n_0$ and each $|j|\leq N$, we can find $w_j\in (2N+1)U_j$ and $z_{j+n} \in B_{j+n}(0,1)$ such that $T_{j-n,j}w_j\in B_{j-n}(0,1)$ and $T_{j,j+n}z_j\in (2N+1)V_j$. Then we have that 
    $$
    w=(\dots,0,0,\frac{w_{-N}}{2N+1},\dots ,\frac{w_{N}}{2N+1},0,0,\dots )\in U
    $$
    and for every $n\ge n_0$,
    $
    B^n_{(T_k)}w\in  B(0,1).
    $
    Similarly,
    $$
    z=(\dots,0,0,\frac{z_{-N+n}}{2N+1},\dots ,\frac{z_{n+N}}{2N+1},0,0,\dots )\in B(0,1)
    $$
    and for every $n\ge n_0$, $B^n_{(T_k)}z\in V$
    which concludes the proof.
\end{proof}

In particular, if each operator $T_k$ is equal to $T$, Condition $(2)$ in Proposition~\ref{prop mixing} is equivalent to require that for every non-empty open set $U$ in $X$, the sets $N_T(B(0,1),U)$ and $N_T(U,B(0,1))$ are cofinite. Since this is equivalent to require that $T$ is mixing (see \cite{GrPe10}), we deduce that in the bilateral case, $B_T$ is mixing if and only if $T$ is mixing.

\begin{cor}\label{bilmix}
On $c_0(X,\mathbb{Z})$ and $\ell^p(X,\mathbb{Z})$ with $1\le p<\infty$, the following assertions are equivalent:
\begin{enumerate}
\item $B_T$ is mixing;
\item $T$ is mixing on $X$.
\end{enumerate}
\end{cor}

In the unilateral case, we can benefit from the dense generalized kernel and adapt the proof of Proposition~\ref{prop mixing} to get the following statement.

\begin{prop}\label{mixing:uni}
Assume that $B_{(T_k)}$ is an operator on $\ell^p((X_k)_{k},\mathbb{N})$ (resp. on $c_0((X_k)_{k},\mathbb{N})$).
The following assertions are equivalent:
\begin{enumerate}[\upshape (1)]
\item $B_{(T_k)}$ is mixing;
%\item For each $U$ non-empty open subset of %$l^p((X_k)_{k\ge 0})$ (resp. of $c_0((X_k)_{k\geq 0})$)
%$$
%N_{B_{(T_k)}}(B(0,1),U) \quad \text{is cofinite;}
%$$
\item for every $j\in \mathbb{N}$, every non-empty open subset $U_j$ of $X_j$, there exists $n_0$ such that for every $n\ge n_0$,
\[T_{j,j+n}(B_{j+n}(0,1))\cap U_j\ne \emptyset.\]
\end{enumerate}
\end{prop}

We immediately get the following characterization for $B_T$.

\begin{cor}\label{mix,uni,T}
On $c_0(X,\mathbb{Z})$ and $\ell^p(X,\mathbb{Z})$ with $1\le p<\infty$, the following assertions are equivalent:
\begin{enumerate}[\upshape (1)]
\item $B_{T}$ is mixing;
\item for any non-empty open set $U$ in $X$,
\[N_{T}(B_X(0,1),U)\quad \text{is cofinite}.\]
\end{enumerate}
In particular, if $T$ is mixing on $X$ then $B_{T}$ is mixing on $\ell^p(X,\mathbb{N})$ and on $c_0(X,\mathbb{N})$.
\end{cor}

Notice that $B_T$ can be mixing even if $T$ is not hypercyclic. It suffices to consider again $T=2Id$ on $\mathbb{K}$.

\subsection{Chaos}

In view of all the above results, we cannot yet deduce if in general, $B_T$ is quasi-conjugate to $T$ in the unilateral context, and if in general, $T$ is quasi-conjugate to $B_T$ in the bilateral one. The study of chaos will provide us with an answer to these two questions.

\begin{prop}\label{chaos-uni}
The operator $B_{(T_k)}$ is chaotic on $\ell^p((X_k)_k,J)$ with $1\le p<\infty$ (resp. on $c_0((X_k)_k,J)$) if and only if for every $k\in J$, there exists a set $\mathcal{D}_k$ in $X_k$ such that $\text{\emph{span}}(\mathcal{D}_k)$ is dense in $X_k$ and such that for every $x_k\in \mathcal{D}_k$, there exists $(x_n)_{n\ge k+1}$ such that $T_{n}x_n=x_{n-1}$ for any $n\ge k+1$ and such that 
\begin{itemize}
\item if $J=\mathbb{N}$,
\[(0,\cdots,0,x_k,x_{k+1},\cdots)\in \ell^p((X_k)_k,\mathbb{N})\ (\text{resp. $c_0((X_k)_k,\mathbb{N})$}),\]
\item if $J=\mathbb{Z}$, 
\[(\cdots,T_{k-2,k}x_k,T_{k-1,k}x_k,x_k,x_{k+1},\cdots)\in \ell^p((X_k)_k,\mathbb{Z})\ (\text{resp. $c_0((X_k)_k,\mathbb{Z})$}).\]
\end{itemize}
Moreover, if $B_{(T_k)}$ is chaotic on $\ell^p((X_k)_k,J)$ (resp. $c_0((X_k)_k,J)$) then $B_{(T_k)}$ is mixing on $\ell^p((X_k)_k,J)$ (resp. $c_0((X_k)_k,J)$).
\end{prop}
\begin{proof}
We perform the proof in the bilateral case. Let us assume that $B_{(T_k)}$ is chaotic. Let $i\in \mathbb{Z}$ and $U_i$ be a non-empty open subset of $X_i$. By density of the periodic points, there exists a periodic point $(x_j)_{j\in \mathbb{Z}}$ for $B_{(T_k)}$ such that $x_i\in U_i$. Let $d$ be the period of $(x_j)_{j\in \mathbb{Z}}$ for $B_{(T_k)}$. We consider the sequence $(y_j)_{j\in\mathbb{Z}}$ defined by
$$
y_j=\begin{cases}
x_j, \quad \text{if} \quad j\in i+d\mathbb{Z},\\
0, \quad \text{otherwise}.
\end{cases}
$$
The sequence $(y_j)_{j\in \mathbb{Z}}$ is thus also a periodic point with period $d$ for $B_{(T_k)}$ and by letting
$$
(z_j)_{j\in \mathbb{Z}}:=(I+B_{(T_k)}+\dots +B_{(T_k)}^{d-1})((y_j)_{j\in \mathbb{Z}}),
$$
we get a fixed point for $B_{(T_k)}$ with $z_i=x_i\in U_i$. We then get the first implication.\\

On the other hand, we first observe that we may assume that $\mathcal{D}_k$ is dense for every $k$ (since the validity of the assumptions for $\mathcal{D}_k$ implies their validity for $\text{span}(\mathcal{D}_k)$). Now, let $U_i$ be a non-empty open set of $X_i$ and $x_i\in \mathcal{D}_i\cap U_i$. There exists a sequence  $(x_n)_{n\ge i+1}$ such that $T_{n}x_n=x_{n-1}$ for any $n\ge i+1$ and such that 
\[y:=(\cdots,T_{i-2,i}x_i,T_{i-1,i}x_i,x_i,x_{i+1},\cdots)\in \ell^p((X_k)_k,\mathbb{Z}))\ (\text{resp. $c_0((X_k)_k,\mathbb{Z})$}).\] For every $d\ge 1$, the sequence $y^{(d)}$ defined by
$$
y_j^{(d)}=\begin{cases}
y_j, \quad \text{if} \quad j\in i+d\mathbb{Z},\\
0, \quad \text{otherwise}
\end{cases}
$$
is a periodic point for $B_{(T_k)}$ of period $d$ and $(y^{(d)})_d$ tends to the sequence $(\dots,0,x_i,0,\dots)$ as $d\rightarrow \infty$ in $\ell^p((X_k)_k,\mathbb{Z}))$ (resp. in $c_0((X_k)_k,\mathbb{Z})$). Since the finitely supported sequences are dense in $\ell^p((X_k)_k,\mathbb{Z}))$ (resp. in $c_0((X_k)_k,\mathbb{Z})$) and the set of periodic points form a linear subspace, we may conclude that $B_{(T_k)}$ has a dense set of periodic points. Finally, we observe that if $U_i$ is a non-empty open subset of $X_i$, we have by considering a fixed point $x$ for $B_{(T_k)}$ with $x_i\in U_i$ that there exists $n_0$ such that for every $n\ge n_0$,
\[T_{i,i+n}B_{i+n}(0,1)\cap U_i\ne \emptyset.\]
In the same way, if $V_i$ is a non-empty open subset of $X_i$, we have by considering a fixed point $x$ for $B_{(T_k)}$ with $x_i\in V_i$ that there exists $n_0$ such that for every $n\ge n_0$,
\[ T_{i-n,i}V_i\cap B_{i-n}(0,1)\ne \emptyset.\]
Hence $B_{(T_k)}$ is hypercyclic and even mixing by Proposition~\ref{prop mixing}.
\end{proof}

\begin{cor}\label{cor-chaos-uni}
The operator $B_{T}$ is chaotic on $\ell^p(X,J)$ with $1\le p<\infty$ (resp. on $c_0(X,J)$) if and only if there exists a set $\mathcal{D}$ in $X$ such that $\text{\emph{span}}(\mathcal{D})$ is dense in $X$ and such that for every $x_1\in \mathcal{D}$, there exists $(x_n)_{n\ge 2}$ such that $T_{n}x_n=x_{n-1}$ for any $n\ge 2$ and such that 
\begin{itemize}
\item if $J=\mathbb{N}$,
\[(x_1,x_{2},\cdots)\in \ell^p(X,\mathbb{N})\ (\text{resp. $c_0(X,\mathbb{N})$}),\]
\item if $J=\mathbb{Z}$, 
\[(\cdots,T^2x_1,Tx_1,x_1,x_{2},\cdots)\in \ell^p(X,\mathbb{Z})\ (\text{resp. $c_0(X,\mathbb{Z})$}).\]
\end{itemize}
\end{cor}

We remark that unlike the properties of hypercyclicity, weak mixing and mixing, chaos cannot always be transferred from $T$ to $B_T$ on $c_0(X,\mathbb{N})$ or $\ell^p(X,\mathbb{N})$.

\begin{prop}
There exists an operator $T$ on a Hilbert space $X$ such that $T$ is chaotic but the only periodic point of $B_T$ in $\ell^p(X,J)$ for any $1\le p<\infty$ or $c_0(X,J)$ is zero. In particular, $T$ is chaotic on $X$ but $B_T$ is not chaotic on $\ell^p(X,J)$ for any $1\le p<\infty$ nor on $c_0(X,J)$.
\end{prop}
\begin{proof}
We know thanks to Badea and Grivaux \cite[Corollary 4.7]{BaGri} that there exists a chaotic operator $S_{\delta}$ on a Hilbert space $X$ such that $\sup_k\|S_{\delta}^{n_k}\|<\infty$ for an increasing sequence $(n_k)$. It implies that for any $d\ge 1$,  $\sup_k\|S_{\delta}^{n_kd}\|<\infty$. If we now assume that $x$ is a periodic point for $B_{S_{\delta}}$ of period $d$ then for any $k,n\in \mathbb{N}$, we get 
\[\|x_n\|=\|S^{n_kd}_{\delta}x_{n+n_kd}\|\le \|S_{\delta}^{n_kd}\| \|x_{n+n_kd}\|\xrightarrow[k\to \infty]{} 0  \]
and thus $x=0$.
\end{proof}

In particular, the above proposition implies that for some operator $T$, the operator \textbf{$B_T$ is not quasi-conjugate to $T$ on $\ell^p(X,\mathbb{N})$ or $c_0(X,\mathbb{N})$}. Obviously, by considering $T=2 \text{Id}$ on $\mathbb{K}$, we also have an example of operator $T$ such that $B_T$ is chaotic on $\ell^p(X,\mathbb{N})$ or $c_0(X,\mathbb{N})$ but $T$ is not. One can wonder if when we consider the bilateral case, under the condition that $B_T$ is chaotic, we can deduce that $T$ is chaotic. We will see that the answer will depend on the value of $p$.

\begin{prop}
There exist a Banach space $X$ and an operator $T$ on $X$ such that $B_T$ is chaotic on $\ell^p(X,\mathbb{Z})$ for any $1<p<\infty$ and on $c_0(X,\mathbb{Z})$ but $T$ is not chaotic on $X$. 
\end{prop}
\begin{proof}
Let $X=\ell^1(\mathbb{K},\mathbb{Z})$ and $T$ the bilateral weighted shift $B_w$ with $w_n=2$ if $n\ge -1$ and $w_n=\frac{n+1}{n}$ otherwise. We then have that for every $n\ge 2$ \[w_{-n}\cdots w_{-2}=1/n.\] The operator $T$ is therefore not chaotic on $X$ since the series $\sum_{n=2}^{\infty} |w_{-n}\cdots w_{-2}|$ is divergent. 
However, we can deduce from Corollary~\ref{cor-chaos-uni} that $B_T$ is chaotic on $\ell^p(X,\mathbb{Z})$ for any $1<p<\infty$ and on $c_0(X,\mathbb{Z})$. Indeed, if we consider $\mathcal{D}=c_{00}(\mathbb{K},\mathbb{Z})$ and $x_1\in \mathcal{D}$ then we can let $x_n=F^{n-1}_v x_1$ for every $n\ge 2$ where $F_v$ is the forward shift associated to the weights $v_n=1/w_{n+1}$ so that $B_w F_v=Id$. It follows that \[(\cdots,B_w^2x_1,B_w x_1,x_1,x_2,\cdots)\in \ell^p(X,\mathbb{Z})\ \text{for every $p>1$}\]
since for $n$ sufficiently big, $\|x_{n+1}\|= \|x_n\|/2$ and $\|B_w^nx_1\|\le C\frac{\|x_1\|}{n}$ for some constant $C>0$.
\end{proof}
We deduce from the previous result that in general, \textbf{$T$ is not quasi-conjugate to $B_T$ on $\ell^p(X,\mathbb{Z})$ for $1<p<\infty$ or on $c_0(X,\mathbb{Z})$}. However, the situation is surprisingly different on $\ell_1(X,\mathbb{Z})$.

\begin{prop}\label{conj-l^1}
Every operator $T$ on $X$ is quasi-conjugate to $B_T$ on $\ell^1(X,\mathbb{Z})$. 
\end{prop}
\begin{proof}
Let $Y=\ell^{1}(X,\mathbb{Z})$. It suffices to consider the continuous map $\phi:Y\to X$ given by
    $$
    \phi:(x_n)_{n\in \mathbb{Z}} \mapsto \sum_{n\in \mathbb{Z}} x_n.
    $$
Since $(x_n)_{n\in \mathbb{Z}}\in \ell^1(X,\mathbb{Z})$, the series $\sum_{n\in \mathbb{Z}} x_n$ is well-defined and the map $\phi$ is actually a linear surjective contraction, which satisfies that $\phi\circ B_{T}=T\circ \phi$.
\end{proof}

We have seen that if $B_{(T_k)}$ is chaotic then $B_{(T_k)}$ is mixing. For the usual unilateral weighted backward shifts acting on $c_0(\mathbb{K},J)$, the converse is also true. We end up this section by showing that this equivalence is not true for $B_T$ on $c_0(X,J)$.  Notice that on $\ell^p(X,J)$ with $1\le p<\infty$, we can easily get such an example by considering $T=B_w$ the unilateral weighted shift with $w_n=(\frac{n+1}{n})^{\frac{1}{p}}$ on $X=\ell^p(\mathbb{K},\mathbb{N})$. Indeed, since $T$ is mixing, $B_T$ is mixing on $\ell^p(X,J)$ (Corollary~\ref{mix,uni,T}) but since $T$ is not chaotic, $B_T$ cannot be chaotic on $\ell^p(X,J)$ by Proposition~\ref{comparison} and Remark~\ref{comprem}. Our approach for $c_0(X,\mathbb{N})$ is similar except that we will not find such a counter-example by considering unilateral weighted shifts on $c_0(\mathbb{K},\mathbb{N})$. However we will find such a counterexample in the family of weighted shifts on directed trees. Such an example can already be found in \cite[Example 9.10 (b)]{GrPa24} under additional assumptions such as never having infinitely many children. We give here another example relying on the contrary on this possibility of trees to have infinitely many branches from a vertex. Moreover, this example will also be used in the proof of Theorem~\ref{kitai Thm}. 

\begin{prop}\label{mixc0}
There exists a rooted directed tree $(V,E)$ and a weighted shift on $c_0(V)$ that is mixing but not chaotic. Consequently, there exists a Banach space $X$ and an operator $T$ on $X$ such that $B_T$ is mixing on $c_0(X,J)$ but not chaotic on $c_0(X,J)$.
\end{prop}
\begin{proof}
We first explain the construction of the tree that we will consider. We start by denoting the root $e^{(1)}_{1}$. We then add a finite branch of each size from the root that we will denote $e^{(1,n_1)}_{j}$ for $1\le j\le n_1$. In other words, we have $\text{Chi}(e^{(1)}_{1})=\{e^{(1,n_1)}_{1}: n_1 \ge 1\}$ and for each $1\le j< n_1$, $\text{Chi}(e^{(1,n_1)}_{j})=\{e^{(1,n_1)}_{j+1}\}$. We continue by adding from each vertex $e^{(1,n_1)}_{n_1}$ a finite branch of each size that we will denote $e^{(1,n_1,n_2)}_{j}$ for $1\le j\le n_2$.
At the end, for each $(1,n_1,\dots,n_k)\in\mathcal{N}$, where \[\mathcal{N}=\{(n_0,n_1,\dots,n_k): n_0=1,\ k\ge 0, \ \text{$n_i \in \mathbb{N}$ for any $1\le i\le k$}\},\]
we get
\[\text{Chi}(e^{(1,n_1,\dots,n_k)}_{n_k})=\{e^{(1,n_1,\dots,n_k,n_{k+1})}_{1}: n_{k+1} \ge 1\}\]
and for all $1\le j< n_k$, \[\text{Chi}(e^{(1,n_1,\dots,n_k)}_{j})=\{e^{(1,n_1,\dots,n_k)}_{j+1}\}.\]

 Given $(1,n_1,\dots,n_k)\in\mathcal{N}$, we consider an integer $N^{(1,n_1,\dots,n_k)}\ge 1$ such that for any $k,n$, 
\[N^{(1,n_1,\dots,n_{k},n)}=N^{(1,n_1,\dots,n_{k},1)}+n-1 \quad \text{and}\quad N^{(1,n_1,\dots,n_k,1)}\ge 2n_k.\]
We can now define our weighted shift $T$ on $c_0(V)$ as follows:
\begin{itemize}
\item $Te^{(1)}_{1}=0$,
\item $Te^{(1,n_1,\dots,n_k)}_j=4e^{(1,n_1,\dots,n_k)}_{j-1}$ if $2\le j\leq n_k$,
\item $Te^{(n_0,n_1,\dots,n_k)}_{1}=2^{-N^{(n_0,n_1,\dots,n_{k})}}e^{(n_0,n_1,\dots,n_{k-1})}_{n_{k-1}}$ where $n_0=1$ and $k\ge 1$.
\end{itemize}
We remark that $T$ is bounded since for any $k \ge 1$ 
\[\sum_{n_k=1}^{\infty}2^{-N^{(1,n_1,\dots,n_k)}}\le 1.\]

We first prove that $T$ is mixing. Let $U_1,U_2$ be two non-empty open sets in $c_0(V)$. We consider $x\in U_1\cap c_{00}(V)$ and $y\in U_2\cap c_{00}(V)$. Since our directed tree has a root, it is clear that there exists $N_0$ such that for any $n\ge N_0$, $T^nx=0$. On the other hand, it is possible to find for each $n$ a vector $S_ny$ such that $T^nS_n y=y$ and $S_n y$ tends to $0$ when $n$ tends to infinity. This can be easily done by relying on a different branch for each $n$. Indeed, if $y=\sum_{(n_0,n_1,\dots,n_k)\in \mathcal{F}}\sum_{j=1}^{n_k} y_{(n_0,n_1,\dots,n_k),j}e^{(n_0,n_1,\dots,n_k)}_j$ for some finite set $\mathcal{F} \subset \mathcal{N}$ then by letting
\begin{align*}
S_n y=& \sum_{(n_0,n_1,\dots,n_k)\in \mathcal{F}}\left(\sum_{j=1}^{n_k-n} \frac{y_{(n_0,n_1,\dots,n_k),j}}{4^n}e^{(n_0,n_1,\dots,n_k)}_{j+n}\right. \\
&\quad\left. +\sum_{j=\max\{1,n_k-n+1\}}^{n_k} \frac{2^{N^{(n_0,n_1,\dots,n_k,n)}}}{4^{n-1}}y_{(n_0,n_1,\dots,n_k),j}e^{(n_0,n_1,\dots,n_k,n)}_{j+n-n_k}\right),
\end{align*}
we can compute that for any $n$, $T^nS_n y=y$ and $S_n y$ tends to $0$ when $n$ tends to infinity since for any $(n_0,n_1,\dots,n_k)\in \mathcal{F}$, we have
\[\frac{2^{N^{(n_0,n_1,\dots,n_k,n)}}}{4^{n-1}}= \frac{2^{N^{(n_0,n_1,\dots,n_k,1)}+n-1}}{4^{n-1}}\to 0.\]
By considering $x+S_n y$, we now deduce that there exists $N_1$ such that for any $n \ge N_1$, $T^n U_1\cap U_2\ne \emptyset$.

Let's show that $T$ is not chaotic. We first remark that $T$ can be seen as an operator $B_{(T_k)}$ defined on $c_0((X_k)_{k},\mathbb{N})$ where $X_k$ is the $c_0$-space supported by the elements in $\text{Chi}^{k-1}(e^{(1)}_{1})$ and $T_k$ is the restriction of $T$ on $X_k$. It follows from Proposition~ \ref{chaos-uni} that if $T$ is chaotic then $T$ possesses a fixed point $x$ such that the coordinate $x^{(1)}_{1}$ has an absolute value bigger than $1$. Assuming that such a vector $x$ exists, we show that for any $(n_0,\dots,n_k)\in \mathcal{N}$, if $|x^{(n_0,\dots,n_k)}_{1}|\ge 1$ then there exists $n\ge 1$ such that $|x^{(n_0,\dots,n_k,n)}_{1}|\ge 1$. This will contradict the fact that $x \in c_0(V)$. If $x$ is a fixed point for $T$ then for any $(n_0,\dots,n_k)\in \mathcal{N}$, $x^{(n_0,\dots,n_k)}_{n_k}=4^{-n_k+1}x^{(n_0,\dots,n_k)}_{1}$ and since
\[x^{(n_0,\dots,n_k)}_{n_k}=\sum_{n=1}^{\infty}T(x^{(n_0,\dots,n_k,n)}_1e^{(n_0,\dots,n_k,n)}_1),\]
there exists $n\ge 1$ such that \[\|T(x^{(n_0,\dots,n_k,n)}_1e^{(n_0,\dots,n_k,n)}_1)\|\ge 2^{-n}4^{-n_k+1}|x^{(n_0,\dots,n_k)}_{1}|.\]
Therefore, if $|x^{(n_0,\dots,n_k)}_{1}|\ge 1$, we conclude that there exists $n\ge 1$ such that
\begin{align*}
|x^{(n_0,\dots,n_k,n)}_{1}|&\ge 2^{N^{(n_0,\dots,n_k,n)}}2^{-n}4^{-n_k+1}|x^{(n_0,\dots,n_k)}_{1}|\\
&\ge 2^{N^{(n_0,\dots,n_k,1)}-2n_k}\ge 1
\end{align*}
by our assumption on $N^{(n_0,\dots,n_k,1)}$. 

The first part of our statement is thus proved and for the second part, it suffices to use Corollary~\ref{mix,uni,T} (or Corollary~\ref{bilmix}) and Proposition~\ref{comparison} (or Remark~\ref{comprem}) to conclude that $B_T$ is mixing but not chaotic on $c_0(c_0(V),J)$.
\end{proof}

\subsection{Frequent hypercyclicity}

For the usual weighted shifts, we can use the characterization obtained for chaos in terms of weights to show that all chaotic weighted shifts are frequently hypercyclic. This is done by using the following criterion \cite{BoGr}.

\begin{theorem}\label{unc-fr-S_n}
Let $X$ be a separable Banach space and $T\in L(X)$.
Suppose that there are a dense subset $X_0$ of $X$ and mappings $S_n:X_0\to X$ such that for all $x\in X_0$, the following assertions hold:
\begin{enumerate}
\item $\sum_{n=0}^k T^kS_{k-n}x$ converges unconditionally in $X$, uniformly in $k$,
\item $\sum_{n=0}^{\infty} T^kS_{k+n}x$ converges unconditionally in $X$, uniformly in $k$,
\item $\sum_{n=0}^{\infty} S_n x$ converges unconditionally in $X$,
\item $T^nS_n x\to x$,
\end{enumerate}
then $T$ is frequently hypercyclic.
\end{theorem}

This link between chaos and frequent hypercyclicity for weighted shifts can be extended to the operators $(B_{(T_k)})$.

\begin{prop}\label{chaosfhc}
On $\ell^p((X_k)_k,J)$ with $1\le p<\infty$ or on $c_0((X_k)_k,J)$, if $B_{(T_k)}$ is chaotic then $B_{(T_k)}$ is frequently hypercyclic.
\end{prop}
\begin{proof} 
We perform the proof for $J=\mathbb{Z}$. Let $Y=\ell^p((X_k)_k,J)$ with $1\le p<\infty$ or $Y=c_0((X_k)_k,J)$.
We show that if $B_{(T_k)}$ is chaotic then we can apply Theorem \ref{unc-fr-S_n}. By Proposition \ref{chaos-uni} and its proof, for each $k\in \mathbb{Z}$, there exists a dense subset $\mathcal{D}_k$ of $X_k$ satisfying that for every $x_k\in \mathcal{D}_k$, there exists $(x_n)_{n\ge k+1}$ such that $T_{n}x_n=x_{n-1}$ for any $n\ge k+1$ and such that 
\[(\cdots,T_{k-2,k}x_k,T_{k-1,k}x_k,x_k,x_{k+1},\cdots)\in Y.\] 
Let $X_0:=c_{00}((\mathcal{D}_k)_k,\mathbb{Z})$. For each $n\geq 0$, for each $x_i\in \mathcal{D}_i$, we set
$$
S_n((\dots ,0,0,\underbrace{x_i}_{i-\text{position}},0,0\dots ))=(\dots ,0,0,\underbrace{x_{i+n}}_{(i+n)-\text{position}},0,0,\dots)
$$
where $x_{i+n}$ is chosen as above. We then extend $S_n$ in a natural way to the finitely supported sequences in $X_0$. It is immediate that $B_{(T_k)}^nS_nx=x$ for every $x\in X_0$. Moreover, for each $x_i\in \mathcal{D}_i$, we have
$$
\sum_{n=1}^{\infty}S_n((\dots ,0,0,\underbrace{x_i}_{i-\text{position}},0,0\dots ))=(\dots ,0,0,0,x_{i+1},x_{i+2},\dots )\in Y,
$$
and we deduce that $\sum_{n=0}^{\infty}S_nx$ converges unconditionally in $Y$ for all $x\in X_0$.
Finally, for each $x_i\in \mathcal{D}_i$ and $l\in \mathbb{N}$, we have
$$
\sum_{n=0}^{l}B^l_{(T_k)}S_{l-n}((\dots ,0,0,\underbrace{x_i}_{i-\text{position}},0,0\dots ))=(\cdots,0,T_{i-l,i}x_i,\dots ,T_{i-1,i}x_i,x_i,0,\cdots)
$$
and
$$
\sum_{n=0}^{\infty}B^l_{(T_k)}S_{l+n}((\dots ,0,0,\underbrace{x_i}_{i-\text{position}},0,0\dots ))=(\dots ,0,x_i,x_{i+1},x_{i+2},\dots ).
$$
From the fact that $(\cdots,T_{k-2,k}x_k,T_{k-1,k}x_k,x_k,x_{k+1},\cdots)\in Y$, we may conclude that the series above converge unconditionally in $Y$, uniformly in $l$ and hence that $B_{(T_k)}$ is frequently hypercyclic. 
\end{proof}

Bayart and Ruzsa showed in \cite{BaRu15} the impressive result that on $\ell^p(\mathbb{K},J)$, a weighted shift $B_w$ is chaotic if and only if $B_w$ is frequently hypercyclic. We can wonder if this equivalence is still true in the context of the operators $B_{(T_k)}$ on $\ell^p((X_k)_k,J)$ with $1\le p<\infty$. The approach of Bayart and Ruzsa allows us to get the following result.

\begin{theorem}\label{fhcrusza}
If $B_{(T_k)}$ is frequently hypercyclic on $\ell_p((X_k)_k,J)$ with $1\le p<\infty$ then for any $k\in J$, any non-empty open set $U_k\subset X_k$,
\[\sum_{n\ge 1}\inf\{\|z\|_{k+n}^p:z\in T^{-1}_{k,k+n}U_k\}<\infty \]
and if $J=\mathbb{Z}$, we also have
\[\sum_{n\ge 1}\inf\{\|z\|_{k-n}^p:z\in T_{k-n,k}U_k\}<\infty.\]
\end{theorem}
\begin{proof}
We again perform the proof for $J=\mathbb{Z}$. Assume that $B_{(T_k)}$ is frequently hypercyclic and that $x\in \ell_p((X_k)_k,\mathbb{Z})$ is a frequently hypercyclic vector for $B_{(T_k)}$. Let $k\in \mathbb{Z}$ and $U_k\subset X_k$ be a non-empty open set. Let $y_k\in U_k\backslash\{0\}$ and $\varepsilon>0$ be such that $\|y_k\|_k\ge 2\varepsilon$ and $B_k(y_k,\varepsilon)\subset U_k$.
Let
\[A=\{n\ge 0:\|B^n_{(T_k)}x-(\dots ,0,0,\underbrace{y_k}_{k-\text{position}},0,0\dots )\|_p< \varepsilon\}.\]
It follows that $\underline{\text{dens}}(A)>0$. If $n\in A$, we have
\begin{align*}
\varepsilon^p&\ge \sum_{m<n}\|T_{k+m-n,k+m}x_{k+m}\|_{k+m-n}^p+\sum_{m>n}\|T_{k+m-n,k+m}x_{k+m}\|_{k+m-n}^p\\
&\ge \sum_{m<n, m\in A}\|T_{k+m-n,k+m}x_{k+m}\|_{k+m-n}^p+\sum_{m>n, m\in A}\|T_{k+m-n,k+m}x_{k+m}\|_{k+m-n}^p
\end{align*}
while 
\begin{align*}
\sum_{m<n, m\in A}\|T_{k+m-n,k+m}x_{k+m}\|_{k+m-n}^p&=\sum_{m<n, m\in A}\|T_{k+m-n,k}T_{k,k+m}x_{k+m}\|_{k+m-n}^p\\
&\ge \sum_{m<n, m\in A}\inf_{z\in B_k(y_k,\varepsilon)}\|T_{k+m-n,k}z\|_{k+m-n}^p
\end{align*}
and
\[\sum_{m>n, m\in A}\|T_{k+m-n,k+m}x_{k+m}\|_{k+m-n}^p\ge \sum_{m>n, m\in A}\inf\{\|z\|_{k+m-n}^p:T_{k,k+m-n}z\in B_k(y_k,\varepsilon)\}.\]
Let $\alpha_n=\inf\{\|z\|_{k+n}^p:T_{k,k+n}z\in B_k(y_k,\varepsilon)\}$ if $n\ge 0$ and $\alpha_n=\inf_{z\in B_k(y_k,\varepsilon)}\|T_{k+n,k}z\|_{k+n}^p$ if $n\le 0$. We remark that $\alpha_0=\inf_{z\in B_k(y_k,\varepsilon)}\|z\|_k^p$ in both cases and that for every $n\geq 1$, we have
\[\{z\in X_{k+n-1}:T_{k,k+n-1}z\in B_k(y_k,\varepsilon)\}\supset \{T_{k+n-1,k+n}z\in X_{k+n-1}:T_{k,k+n}z\in B_k(y_k,\varepsilon)\}\]
and thus
\begin{align*}
\alpha_n&=\inf\{\|z\|_{k+n}^p:T_{k,k+n}z\in B_k(y_k,\varepsilon)\}\\
&\ge \|T_{k+n-1,k+n}\|^{-p}\inf\{\|T_{k+n-1,k+n}z\|_{k+n}^p:T_{k,k+n}z\in B_k(y_k,\varepsilon)\}\\
&\ge \|B_{(T_k)}\|^{-p}\inf\{\|z\|_{k+n-1}^p:T_{k,k+n-1}z\in B_k(y_k,\varepsilon)\}= \|B_{(T_k)}\|^{-p}\alpha_{n-1}.
\end{align*}
Moreover, for every $n\le 0$, we have
\begin{align*}
\alpha_n&=\inf_{z\in B_k(y_k,\varepsilon)}\|T_{k+n,k}z\|_{k+n}^p\\
& \ge \|T_{k+n-1,k+n}\|^{-p}\inf_{z\in B_k(y_k,\varepsilon)}\|T_{k+n-1,k}z\|_{k+n-1}^p\\
& \ge \|B_{(T_k)}\|^{-p}\alpha_{n-1}.
\end{align*}
Let $\beta_n=\sum_{m\in A}\alpha_{m-n}$. We get that
\begin{align*}
\beta_n&=\sum_{m<n, m\in A}\inf_{z\in B_k(y_k,\varepsilon)}\|T_{k+m-n,k}z\|_{k+m-n}^p+\alpha_0\\
&\quad\quad+\sum_{m>n, m\in A}\inf\{\|z\|_{k+m-n}^p:T_{k,k+m-n}z\in B_k(y_k,\varepsilon)\}\le \varepsilon^p+\alpha_0
\end{align*}
In other words, the sequence $(\beta_n)$ is bounded and it follows from Corollary 9 in \cite{BaRu15} that $\sum_{n\in \mathbb{Z}}\alpha_n$ is convergent, i.e.
\[\sum_{n>0}\inf\{\|z\|_{k+n}^p:T_{k,k+n}z\in B_k(y_k,\varepsilon)\}<\infty \quad\text{and}\quad \sum_{n>0}\inf_{z\in B_k(y_k,\varepsilon)}\|T_{k-n,k}z\|^p<\infty.\]
Finally, since $B_k(y_k,\varepsilon)\subset U_k$, we get the desired result.
\end{proof}

This result allows us to deduce that the frequent hypercyclicity of $B_{(T_k)}$ on $\ell_p((X_k)_k,J)$ implies that $B_{(T_k)}$ is mixing.

\begin{cor}\label{fhcmix}
If $B_{(T_k)}$ is frequently hypercyclic on $\ell_p((X_k)_k,J)$ with $1\le p <\infty$ then $B_{(T_k)}$ is mixing on $\ell_p((X_k)_k,J)$.
\end{cor}
\begin{proof}
The result follows from Theorem~\ref{fhcrusza} and Proposition~\ref{prop mixing} (or Proposition~\ref{mixing:uni}). It suffices to remark that for every $n\ge 1$ and every $U_j\subset X_j$ non-empty open set,
\[T_{j,j+n}B_{j+n}(0,1)\cap U_j\ne \emptyset \quad \Leftrightarrow  \quad \inf\{\|z\|_{j+n}^p:z\in T^{-1}_{j,j+n}U_j\}<1\]
and that 
\[T_{j-n,j}U_j\cap B_{j-n}(0,1)\ne \emptyset \quad \Leftrightarrow  \quad 
\inf\{\|z\|_{j-n}^p:z\in T_{j-n,j}U_j\}<1.\]
\end{proof}

Notice that this implication is not true in general on $c_0((X_k)_k,J)$ since even on $c_0(\mathbb{K},\mathbb{N})$ we can find a frequently hypercyclic weighted shift that is not mixing (\cite{BaGr07}). If we want to extend the result of Bayart and Ruzsa for weighted shifts on $\ell^p$, it will be nice if the convergences obtained in Theorem~\ref{fhcrusza} imply chaos. However, this will not be the case because these conditions are too weak. In fact, these conditions are not even equivalent to frequent hypercyclicity in the case of $B_T$.

\begin{prop}\label{proptree}
Let $1\le p<\infty$. There exist a rooted directed tree $V$ and a weighted shift $T$ on $X=\ell^p(V)$ (resp. on $X=c_0(V)$) such that for any non-empty open set $U\subset X$,
\[\sum_{n\ge 1}\inf\{\|z\|_{X}:z\in T^{-n}U\}<\infty \quad\text{and}\quad \sum_{n\ge 1}\inf\{\|z\|_{X}:z\in T^n U\}<\infty.\]
but $B_T$ is not frequently hypercyclic on $\ell^p(X,J)$ (resp. on $c_0(X,J)$).
\end{prop}
\begin{proof}
Let $\mathcal{N}=\{(n_0,n_1,\dots,n_k): n_0=1,\ k\ge 0\ \text{and $n_j\in \mathbb{N}$ for any $1\le j\le k$}\}$ and a family  $(l^{(1,n_1,\dots,n_k)})_{(1,n_1,\dots,n_k)\in\mathcal{N}}\subset  \mathbb{N}$. We consider the rooted directed tree $V$ where $e^{(1)}_1$ is the root and where for each $(1,n_1,\dots,n_k)\in\mathcal{N}$, 
\[\text{Chi}(e^{(1,n_1,\dots,n_k)}_{l^{(1,n_1,\dots,n_k)}})=\{e^{(1,n_1,\dots,n_k,n_{k+1})}_{1}: n_{k+1} \ge 1\}\]
and for all $1\le j< l^{(1,n_1,\dots,n_k)}$, \[\text{Chi}(e^{(1,n_1,\dots,n_k)}_{j})=\{e^{(1,n_1,\dots,n_k)}_{j+1}\}.\]
This tree is similar to the tree considered in the proof of Proposition~\ref{mixc0} excepted that the length of the segment $(e^{(1,n_1,\dots,n_k)}_{j})_{1\le j\le l^{(1,n_1,\dots,n_k)}}$ is now given by the parameter $l^{(1,n_1,\dots,n_k)}$.\\

Let $X=\ell^p(V)$ (resp. $c_0(V)$), $a^{(1)}=1$ and  $l^{(1)}=1$. Given $(1,n_1,\dots,n_k)\in\mathcal{N}$ with $k\ge 1$, we consider 
\[a^{(1,n_1,\dots,n_k)}=(n_1+\dots+n_k+2(l^{(1,n_1,\dots,n_{k-1})}-a^{(1,n_1,\dots,n_{k-1})})+3)^2,\]
\[N^{(1,n_1,\dots,n_k)}= n_k+2(l^{(1,n_1,\dots,n_{k-1})}-a^{(1,n_1,\dots,n_{k-1})})+1\]
and finally
\[l^{(1,n_1,\dots,n_k)}=a^{(1,n_1,\dots,n_k+1)}+N^{(1,n_1,\dots,n_k+1)}+n_k+1\]
so that for any $(1,n_1,\dots,n_k)\in\mathcal{N}$ 
\begin{multline}
\bigcup_{n\ge 1} [a^{(1,n_1\dots,n_k,n)}+N^{(1,n_1,\dots,n_k,n)}+n,l^{(1,n_1,\dots,n_k,n)}]\\
\quad=[a^{(1,n_1\dots,n_k,1)}+N^{(1,n_1,\dots,n_k,1)}+1,\infty[.
\label{cofinite}
\end{multline}

We can now consider the weighted shift $T$ on $X$ given by
\begin{itemize}
\item $Te^{(1)}_{1}=0$
\item $Te^{(n_0,\dots,n_k)}_{1}=2^{-N^{(n_0,\dots,n_k)}}e^{(n_0,\dots,n_{k-1})}_{l^{(n_0,\dots,n_{k-1})}}$ where $n_0=1$ and $k\ge 1$,
\item $Te^{(1,n_1,\dots,n_k)}_j=e^{(1,n_1,\dots,n_k)}_{j-1}$ if $1< j\le a^{(1,n_1,\dots,n_{k})}$ and $k\ge 1$,
\item $Te^{(1,n_1,\dots,n_k)}_j=4e^{(1,n_1,\dots,n_k)}_{j-1}$ if $a^{(1,n_1,\dots,n_{k})}< j\le l^{(1,n_1,\dots,n_{k})}$ and $k\ge 1$.
\end{itemize}
We already remark that for every non-empty open set $U\subset X$, $\inf\{\|z\|_{X}:z\in T^n U\}$ is ultimately equal to $0$ since finitely supported sequences are dense in $X$ and have an orbit eventually equal to $0$. We now show that for every non-empty open set $U\subset X$,
\[\sum_{m\ge 1}\inf\{\|z\|_{X}:z\in T^{-m}U\}<\infty.\]
As in the proof of Proposition~\ref{mixc0}, the idea will be to select a convenient branch depending on the considered iterate. Let $U$ be a non-empty open set in $X$. Since the finitely supported sequences are dense, we can find a finite set $\mathcal{F}\subset \mathcal{N}$ and
\[x=\sum_{(n_0,n_1,\dots,n_k)\in \mathcal{F}}\sum_{j=1}^{l^{(n_0,n_1,\dots,n_{k})}} x^{(n_0,n_1,\dots,n_k)}_je^{(n_0,n_1,\dots,n_k)}_j \in U\]
so that for every $m\ge 1$,
\begin{align*}
&\inf\{\|z\|_{X}:z\in T^{-m}U\}\\
&\quad \le \sum_{(n_0,n_1,\dots,n_k)\in \mathcal{F}}\sum_{j=1}^{l^{(n_0,n_1,\dots,n_{k})}} |x^{(n_0,n_1,\dots,n_k)}_j| \inf\{\|z\|_{X}:z\in T^{-m}\{e^{(n_0,n_1,\dots,n_k)}_j\}\}.
\end{align*} Since $\mathcal{F}$ is finite, it is then enough to show that for any $(n_0,n_1,\dots,n_k)\in \mathcal{N}$ and any $1\le j\le l^{(n_0,n_1,\dots,n_{k})}$,
\[\sum_{m \ge 1}\inf\{\|z\|_{X}:z\in T^{-m}\{e^{(n_0,n_1,\dots,n_k)}_j\}\}<\infty.\]
Let $(n_0,n_1,\dots,n_k)\in \mathcal{N}$ and $1\le j\le l^{(n_0,n_1,\dots,n_{k})}$. Let $m\ge l^{(n_0,n_1,\dots,n_k)}-j+a^{(n_0,n_1,\dots,n_k,1)}+1+N^{(n_0,\dots,n_k,1)}$. By \eqref{cofinite}, there exists $n\ge 1$ such that
\[l^{(n_0,n_1,\dots,n_k)}-j+a^{(n_0,n_1,\dots,n_k,n)}+n+N^{(n_0,\dots,n_k,n)}\le m\le l^{(n_0,n_1,\dots,n_k)}-j+l^{(n_0,n_1,\dots,n_k,n)}.\]
We then let 
\[z=\frac{2^{N^{(n_0,\dots,n_k,n)}}}{4^{l^{(n_0,n_1,\dots,n_k)}-\max\{a^{(n_0,n_1,\dots,n_k)},j\}} 4^{j+m-l^{(n_0,n_1,\dots,n_k)}-a^{(n_0,n_1,\dots,n_k,n)}}}e^{(n_0,n_1,\dots,n_k,n)}_{j+m-l^{(n_0,n_1,\dots,n_k)}}\] so that
\[T^mz=e^{(n_0,n_1,\dots,n_k)}_j \quad \text{and}\quad \|z\|\le \frac{2^{N^{(n_0,\dots,n_k,n)}}}{4^{j+m-l^{(n_0,n_1,\dots,n_k)}-a^{(n_0,n_1,\dots,n_k,n)}}}.\]
We deduce that for any $n\ge 1$
\begin{align*}
&\sum_{m=l^{(n_0,n_1,\dots,n_k)}-j+a^{(n_0,n_1,\dots,n_k,n)}+n+N^{(n_0,\dots,n_k,n)}}^{l^{(n_0,n_1,\dots,n_k)}-j+l^{(n_0,n_1,\dots,n_k,n)}}
\inf\{\|z\|_{X}:z\in T^{-m}\{e^{(n_0,n_1,\dots,n_k)}_j\}\}\\
&\le \sum_{m=l^{(n_0,n_1,\dots,n_k)}-j+a^{(n_0,n_1,\dots,n_k,n)}+n+N^{(n_0,\dots,n_k,n)}}^{l^{(n_0,n_1,\dots,n_k)}-j+l^{(n_0,n_1,\dots,n_k,n)}} \frac{2^{N^{(n_0,\dots,n_k,n)}}}{4^{j+m-l^{(n_0,n_1,\dots,n_k)}-a^{(n_0,n_1,\dots,n_k,n)}}}\\
& \le \frac{2^{N^{(n_0,\dots,n_k,n)}}}{4^{n+N^{(n_0,\dots,n_k,n)}-1}}\le \frac{1}{4^{n-1}}. 
\end{align*}
By \eqref{cofinite}, it follows that \[\sum_{m \ge 1}\inf\{\|z\|_{X}:z\in T^{-m}\{e^{(n_0,n_1,\dots,n_k)}_j\}\}<\infty.\]

It remains to show that $B_T$ is not frequently hypercyclic on $\ell^p(X,J)$ (resp. on $c_0(X,J)$). The strategy consists in showing that if $x$ is hypercyclic for $B_T$ then the norm $\|B_T^{n}x\|$ is too often bigger than $1$ by relying on the fact that the ratio $\frac{a^{(1,n_1,\dots,n_k)}}{l^{(1,n_1,\dots,n_k)}}$ is tending to $1$ when $n_1+\cdots+n_k$ tends to infinity. Let $y=(y_j)_{j\in J}$ given by $y_1=2e^{(1)}_{1}$ and $y_j=0$ for any $j\in J\backslash\{1\}$. If $x=(x_j)_{j\in J}\in \ell_p(X,J)$ (resp. on $c_0(X,J)$) with $\|x\|< 1$ satisfies
\[\|B^n_T x-y\|<1\]
then we have $\|T^nx_{n+1}-2e^{(1)}_{1}\|<1$. It follows that there exists
$(1,n_1,\dots,n_k)\in \mathcal{N}$ such that \[\|T^n x^{(1,n_1,\cdots,n_k)}_{n+1,j}e^{(1,n_1,\cdots,n_k)}_{j}\|\ge 2^{-n_1-\cdots-n_k-k}\] and $n=l^{(1,n_1)}+\cdots+l^{(1,n_1,\cdots,n_{k-1})}+j$ with $1\le j\le l^{(1,n_1,\cdots,n_k)}$. Since for any $l\ge 1$
\[N^{(n_0,n_1,\dots,n_l)}= n_l+2(l^{(n_0,n_1,\dots,n_{l-1})}-a^{(n_0,n_1,\dots,n_{l-1})})+1,\]
we can compute that
\begin{align*}
\|T^n e^{(1,n_1,\cdots,n_k)}_{j}\|&=\frac{4^{(l^{(1,n_1)}-a^{(1,n_1)})+\cdots+(l^{(1,n_1,\cdots,n_{k-1})}-a^{(1,n_1,\cdots,n_{k-1})})+\max(j-a^{(1,n_1,\dots,n_k)},0)}}{2^{N^{(1,n_1)}+\cdots +N^{(1,n_1,\dots,n_k)}}}\\
&= \frac{4^{\max(j-a^{(1,n_1,\dots,n_k)},0)}}{2^{n_1+\cdots+n_k+k}}.
\end{align*}
Therefore, since $|x^{(1,n_1,\cdots,n_k)}_{n+1,j}|< 1$ and $\|T^n x^{(1,n_1,\cdots,n_k)}_{n+1,j}e^{(1,n_1,\cdots,n_k)}_{j}\|\ge 2^{-n_1-\cdots-n_k-k}$, we deduce that $j> a^{(1,n_1,\cdots,n_k)}$ 
and that
\[|x^{(1,n_1,\cdots,n_k)}_{n+1,j}|\ge 4^{-(j-a^{(1,n_1,\dots,n_k)})}.\]
It follows that for any $j-a^{(1,n_1,\dots,n_k)}\le m< j$, we have
\[\|T^m x^{(1,n_1,\cdots,n_k)}_{n+1,j}e^{(1,n_1,\cdots,n_k)}_{j}\|\ge 1\]
and thus $\|B_T^m x\|\ge 1$.
Since $j\le l^{(1,n_1,\cdots,n_k)}$, we get
\[\frac{|\{m\le l^{(1,n_1,\cdots,n_k)}:\|B_T^mx\|\ge 1\}|}{l^{(1,n_1,\cdots,n_k)}}\ge \frac{a^{(1,n_1,\dots,n_k)}}{l^{(1,n_1,\dots,n_k)}}.\]
By definition of our parameters and if we let $C^{(n_0,\dots,n_{k-1})}=l^{(n_0,\dots,n_{k-1})}-a^{(n_0,\dots,n_{k-1})}$, we have
\begin{align*}
\frac{a^{(1,n_1,\dots,n_k)}}{l^{(1,n_1,\dots,n_k)}}&=\frac{a^{(1,n_1,\dots,n_k)}}{a^{(1,n_1,\dots,n_k +1)}+2n_k+2C^{(n_0,\dots,n_{k-1})}+3}\\
&=\frac{(n_1+\dots+n_k+2C^{(n_0,\dots,n_{k-1})}+3)^2}{(n_1+\dots+n_k+1+2C^{(n_0,\dots,n_{k-1})}+3)^2+ 2n_k+2C^{(n_0,\dots,n_{k-1})}+3}
\end{align*}
and thus
\begin{align*}
1&\le \frac{l^{(1,n_1,\dots,n_k)}}{a^{(1,n_1,\dots,n_k)}}\\
&=  \frac{(n_1+\dots+n_k+1+2C^{(n_0,\dots,n_{k-1})}+3)^2}{(n_1+\dots+n_k+2C^{(n_0,\dots,n_{k-1})}+3)^2}+\frac{2n_k+2C^{(n_0,\dots,n_{k-1})}+3}{(n_1+\dots+n_k+2C^{(n_0,\dots,n_{k-1})}+3)^2}\\
&\le \frac{(n_1+\dots+n_k+1)^2}{(n_1+\dots+n_k)^2}+\frac{2}{n_1+\dots+n_k}.
\end{align*}
Let $\varepsilon>0$ and $x$ a hypercyclic vector for $B_T$ with $\|x\|< 1$. Then there exist infinitely many $n$ such that $\|B^n_T x-y\|<\frac{1}{2}$ and thus infinitely many $(1,n_1,\cdots,n_k)$ such that 
\[\frac{|\{m\le l^{(1,n_1,\cdots,n_k)}:\|B_T^mx\|\ge 1\}|}{l^{(1,n_1,\cdots,n_k)}}\ge \frac{a^{(1,n_1,\dots,n_k)}}{l^{(1,n_1,\dots,n_k)}}.\]
In particular, among these ones, there exists $(1,n_1,\cdots,n_k)$ such that $n_1+\dots+n_k$ is arbitrarily big and thus such that
\[\frac{|\{m\le l^{(1,n_1,\cdots,n_k)}:\|B_T^mx\|\ge 1\}|}{l^{(1,n_1,\cdots,n_k)}}\ge 1-\varepsilon.\]
This implies that $\underline{\text{dens}}\{m\ge 1:\|B_T^mx\|< 1\}=0$ and that $x$ is not frequently hypercyclic for $B_T$. Therefore we can conclude that $B_T$ is not frequently hypercyclic.
\end{proof}
\begin{remark}
We can also deduce from the previous proposition that there exist a Banach space $X$ and an operator $T$ on $X$ such that $B_T$ is mixing but not frequently hypercyclic on $c_0(X,J)$.
\end{remark}

We can now wonder if the proof of Bayart-Ruzsa for weighted shifts on $\ell^p(\mathbb{K},J)$ does not work in our context because the conditions obtained in Theorem~\ref{fhcrusza} are too weak or because frequent hypercyclicity itself is too weak to imply chaos. We show that we have no hope to extend the equivalence between frequent hypercyclicity and chaos to the operators $B_T$ on $\ell^p(X,J)$.

\begin{prop}\label{fhcnotchaos}
Let $1\le p<\infty$. There exist a rooted directed tree $V$ and a weighted shift $T$ on $X=\ell^1(V)$ such that $B_T$ is frequently hypercyclic but not chaotic on $\ell^p(X,J)$. 
\end{prop}
\begin{proof}
As in the proof of Proposition~\ref{proptree}, we consider $\mathcal{N}=\{(n_0,n_1,\dots,n_k): n_0=1,\ k\ge 0\ \text{and $n_j\in \mathbb{N}$ for any $1\le j\le k$}\}$. We let $l^{(1,n_1,\dots,n_k)}=n_k+2$ for any $(1,n_1,\dots,n_k)\in\mathcal{N}$ and consider the rooted directed tree $V$ where $e^{(1)}_1$ is the root and where for each $(1,n_1,\dots,n_k)\in\mathcal{N}$, 
\[\text{Chi}(e^{(1,n_1,\dots,n_k)}_{l^{(1,n_1,\dots,n_k)}})=\{e^{(1,n_1,\dots,n_k,n_{k+1})}_{1}: n_{k+1} \ge 1\}\]
and for all $1\le j< l^{(1,n_1,\dots,n_k)}$, \[\text{Chi}(e^{(1,n_1,\dots,n_k)}_{j})=\{e^{(1,n_1,\dots,n_k)}_{j+1}\}.\]

Let $N^{(1,n_1,\dots,n_k)}=2l^{(1,n_1,\dots,n_k)}+n_k$ for any $(1,n_1,\dots,n_k)\in\mathcal{N}$. We consider the weighted shift $T$ on $X= \ell^1(V)$ given by
\begin{itemize}
\item $Te^{(1)}_1=0$,
\item $Te^{(n_0,\dots,n_k)}_1=4e^{(n_0,\dots,n_{k-1})}_{l^{(n_0,\dots,n_{k-1})}}$ when $n_0=1$ and $k\ge 1$,
\item $Te^{(1,n_1,\dots,n_k)}_j=4e^{(1,n_1,\dots,n_k)}_{j-1}$ if $1< j<l^{(1,n_1,\dots,n_k)}$,
\item $Te^{(1,n_1,\dots,n_k)}_{l^{(1,n_1,\dots,n_k)}}=2^{-N^{(1,n_1,\dots,n_{k})}}e^{(1,n_1,\dots,n_k)}_{l^{(1,n_1,\dots,n_k)}-1}$.
\end{itemize}
We notice that our choice of $N^{(1,n_1,\dots,n_k)}$ implies that the product of weights along the segment $(e^{(1,n_1,\dots,n_k)}_j)_{1\le j\le l^{(1,n_1,\dots,n_k)}}$ is smaller than $1$. We will use this fact to prove that $B_T$ has no fixed point (excepted $0$).\\

We first show that $B_T$ is frequently hypercyclic on $\ell^p(X,J)$ by using Theorem~\ref{unc-fr-S_n}. Let $X_0=c_{00}(c_{00}(V),J)$. Let $m\in J$, $(1,n_1,\dots,n_k)\in \mathcal{N}$ and $1\le j\le l^{(1,n_1,\dots,n_k)}$. It suffices to show that Theorem~\ref{unc-fr-S_n} is satisfied for any $x\in X_0$ given by $x_m= e^{(1,n_1,\dots,n_k)}_j$ and $x_{m'}=0$ if $m'\ne m$. To this end, we let $S_0x=x$, and for any $n\ge 1$, $S_n x=y$ with $y_{m'+n}=0$ if $m'\ne m$ and 
\[y_{m+n} = \left\{
    \begin{array}{ll}
        4^{-n}e^{(1,n_1,\dots,n_k)}_{j+n} & \mbox{if } j+n<l^{(1,n_1,\dots,n_k)} \\
        4^{-(n-1)}2^{N^{(1,n_1,\dots,n_{k})}}e^{(1,n_1,\dots,n_k)}_{j+n} & \mbox{if } j+n=l^{(1,n_1,\dots,n_k)}\\
        4^{-(n-1)}2^{N^{(1,n_1,\dots,n_{k})}}e^{(1,n_1,\dots,n_k,j+n-l^{(1,n_1,\dots,n_k)})}_{j+n-l^{(1,n_1,\dots,n_k)}}& \mbox{if } j+n>l^{(1,n_1,\dots,n_k)}\\ 
    \end{array}
\right.\]
which is well-defined since for any $m\ge 1$, $m<l^{(1,n_1,\dots,n_k,m)}$. We remark that in the third case, our choice of the branch depends on $n$ and that for every $n$, we have $B_T^n S_nx=x$. We can therefore deduce that the operator $B_T$ is frequently hypercyclic since
\begin{enumerate}
\item $\sum_{n=0}^{K} B_T^KS_{K-n} x=\sum_{n=0}^K B_T^n x$ converges uniformly in $K$ because$B_T^n x=0$ for all $n\geq j+l^{(1,\dots ,n_{k-1})}+\dots +l^{(1)}$;
\item $\sum_{n=0}^{\infty}B_T^KS_{K+n}x$ converges unconditionally in $\ell^p(X,J)$ uniformly in $K$ because for all $n> l^{(1,n_1,\dots ,n_k)}-j$, $\|B^KS_{K+n}x\|=2^{N^{(1,n_1,\dots ,n_k)}}4^{-(n-1)}$;
\item $\sum_{n=0}^{\infty} S_n x$ converges unconditionally because \[\sum_{n=1}^{\infty} \|S_n x\|\le \sum_{n=0}^{\infty}4^{-(n-1)}2^{N^{(1,n_1,\dots,n_k)}}<\infty;\]
\end{enumerate}

We now show that $B_T$ is not chaotic on $\ell^p(X,J)$. Assume that $B_T$ is chaotic. By Proposition~\ref{chaos-uni}, there then exists a fixed point $x\in \ell^p(X,J)$ for $B_T$ such that $|x_{l^{(1)},l^{(1)}}^{(1)}|\ge 1$. We show that for any $(1,n_1,\dots,n_{k-1})\in \mathcal{N}$ if $|x_{n,l^{(1,n_1,\dots,n_{k-1})}}^{(1,n_1,\dots,n_{k-1})}|\ge 1$, where $n=l^{(1)}+\dots +l^{(1,n_1,\dots ,n_{k-1})}$, then there exists $n_k$ such that \[|x_{n+l^{(1,n_1,\dots ,n_k)},l^{(1,n_1,\dots,n_{k})}}^{(1,n_1,\dots,n_{k})}|\ge 1.\] This will be a contradiction with $x \in \ell^p(X,J)$.
Indeed, if $x$ is a fixed point for $B_T$ satisfying $|x_{n,l^{(1,n_1,\dots,n_{k-1})}}^{(1,n_1,\dots,n_{k-1})}|\ge 1$, there exists $n_k\ge 1$ such that 
\[\|T(x_{n+1,1}^{(1,n_1,\dots,n_{k})}e_{1}^{(1,n_1,\dots,n_{k})})\|\ge 2^{-n_k}\]
and thus such that $4|x_{n+1,1}^{(1,n_1,\dots,n_{k})}|\ge 2^{-n_k}$.
Since $x$ is a fixed point, we then also have
\[|x_{n+l^{(1,n_1,\dots ,n_k)},l^{(1,n_1,\dots,n_{k})}}^{(1,n_1,\dots,n_{k})}|\ge 4^{-l^{(1,n_1,\dots,n_{k})}+2}2^{N^{(1,n_1,\dots,n_{k})}}2^{-n_k-2}.\]
Therefore, since $N^{(1,n_1,\dots,n_{k})}=2l^{(1,n_1,\dots,n_{k})}+n_k$, we get $|x_{n+l^{(1,n_1,\dots ,n_k)},l^{(1,n_1,\dots,n_{k})}}^{(1,n_1,\dots,n_{k})}|\ge 1$. By induction, we deduce that $x$ will have infinitely many coordinates with a norm bigger than $1$. Contradiction.
\end{proof}

\begin{remark}
  Let us mention that the weighted shift $T$ on the rooted tree $V$ of the previous proposition, satisfies that it is frequently hypercyclic but not chaotic on $\ell^1(V)$. Indeed, the fact that $T$ is frequently hypercyclic follows immediately from Proposition~\ref{fhcnotchaos} (for $p=1$) and Proposition~\ref{conj-l^1}. The argument that $T$ is not chaotic is similar to the  proof that $B_T$ is not chaotic.
\end{remark}

We conclude from all the previous results that we have the following implications between the five investigated dynamical properties for the operators $(B_{(T_k)})$ and that no other implication is true on $\ell^p((X_k)_k,J)$ or on $c_0((X_k)_k,J)$ (even if we restrict ourselves to the operators $B_T$) as mentioned in the introduction.

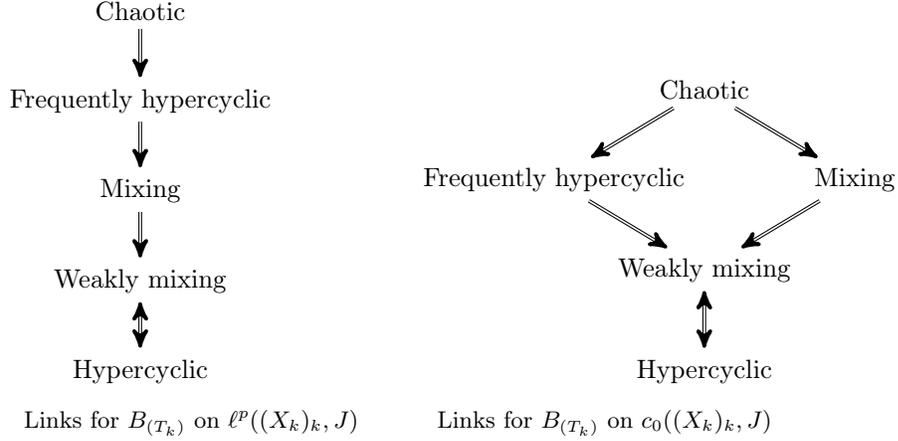
\begin{figure}[H]

\begin{subfigure}{.4\textwidth}
      \begin{tikzpicture}[xscale=1,yscale=1.5,>=stealth',shorten >=1pt]
%        \everymath{\scriptstyle}
       
        \path (-2,5) node[] (q1) {Chaotic};
        \path (-2,4.2) node[] (q0) {Frequently hypercyclic};
        \path (-2,3.4) node[] (q2) {Mixing};

        \path (-2,2.6) node[] (q3) {Weakly mixing};
        \path (-2,1.8) node[] (q4) {Hypercyclic};

		\draw[double,arrows=->] (q1) -- (q0);
        \draw[double,arrows=->] (q0) -- (q2);
        \draw[double,arrows=->] (q2) -- (q3);
        \draw[double,arrows=<->] (q3) -- (q4);

        \end{tikzpicture}
    \caption*{Links for $B_{(T_k)}$ on $\ell^p((X_k)_k,J)$}
\end{subfigure}
\hspace*{0.2cm}
\begin{subfigure}{.4\textwidth}
      \begin{tikzpicture}[xscale=1,yscale=1.5,>=stealth',shorten >=1pt]
%        \everymath{\scriptstyle}
       
        \path (6,3.4) node[] (q2) {Mixing};
        \path (4,4.2) node[] (q1) {Chaotic};
        \path (2,3.4) node[] (q0) {Frequently hypercyclic};

        \path (4,2.6) node[] (q3) {Weakly mixing};
        \path (4,1.7) node[] (q4) {Hypercyclic};

		\draw[double,arrows=->] (q1) -- (q0);
        \draw[double,arrows=->] (q1) -- (q2);
         \draw[double,arrows=->] (q0) -- (q3);        
         \draw[double,arrows=->] (q2) -- (q3);
        \draw[double,arrows=<->] (q3) -- (q4);

        \end{tikzpicture}
    \caption*{Links for $B_{(T_k)}$ on $c_0((X_k)_k,J)$}
\end{subfigure}
\end{figure}

\section{Kitai Criterion and Frequent Hypercyclicity Criterion}\label{Criteria}

In Linear dynamics, an important way to deduce some dynamical properties for an operator consists in showing that the operator possesses a dense set of orbits tending to $0$ (more or less rapidly) and a dense set of vectors with "backward" orbit tending to $0$ (more or less rapidly). Two important examples of such criteria are the Kitai Criterion and the Frequent Hypercyclicity Criterion. We show in this section how these criteria are related to the dynamical behavior of $B_T$.

\subsection{Kitai Criterion}

The criterion given by Kitai is the following (see \cite{Kit}). 

\begin{theorem}[Kitai Criterion]
If there are dense subsets $X_0,Y_0$ in $X$ and a map $S:Y_0\to Y_0$ such that
\begin{enumerate}
\item $T^nx\to 0$ for each $x\in X_0$,
\item $S^ny\to 0$ for each $y\in Y_0$,
\item $TSy=y$ for each $y\in Y_0$,
\end{enumerate}
then $T$ is mixing.
\end{theorem}

However, it is well-known that we can replace the map $S$ and its iterates by a sequence of maps $(S_n)$ (see \cite{BePe}). In doing so, we get the Hypercyclicity Criterion along the whole sequence $(n)$.

\begin{theorem}[Hypercyclicity Criterion along $(n)$]\label{kit5}
If there are dense subsets $X_0,Y_0$ in $X$ and maps $S_n:Y_0\to X$ such that
\begin{enumerate}
\item $T^nx\to 0$ for each $x\in X_0$,
\item $S_ny\to 0$ for each $y\in Y_0$,
\item $T^nS_ny\to y$ for each $y\in Y_0$,
\end{enumerate}
then $T$ is mixing.
\end{theorem}

These criteria do not characterize mixing operators. Indeed, it was shown by Grivaux in \cite{Gri}, that there exists a mixing operator $T$ such that for any non-zero vector $x$, $T^nx$ does not tend to $0$. We investigate the links between the dynamical properties of $B_T$ and these two criteria (with the additional assumption that $X_0=Y_0$ or not).

\begin{theorem}\label{kitai Thm}
Let $X$ be a separable Banach space and $T\in L(X)$.
Given the following assertions : 
\begin{enumerate}
\item $B_T$ is chaotic on $c_0(X,\mathbb{Z})$,
\item $T$ satisfies the Kitai Criterion with $X_0=Y_0$,
\item $T$ satisfies the Kitai Criterion,
\item $T$ satisfies the Hypercyclicity Criterion along $(n)$ with $X_0=Y_0$,
\item $T$ satisfies the Hypercyclicity Criterion along $(n)$,
\item $T$ is mixing,
\item $B_T$ is mixing on $c_0(X,\mathbb{Z})$,
\end{enumerate}
then 
\[(1)\Leftrightarrow (2) \Rightarrow (3) \Rightarrow (4) \Leftrightarrow (5) \Rightarrow (6) \Leftrightarrow (7).\]
and the other implications are false in general.
\end{theorem}
\begin{proof}\text{}\\
$(1)\Rightarrow (2)$. If we denote by $\text{Fix}(B_T)$ the set of fixed points for $B_T$, it follows from Proposition~\ref{chaos-uni} that there exists a sequence $(z_n)_{n\in \mathbb{N}}\subset \text{Fix}(B_T)\backslash\{0\}$ such that $\{z_{n,m}\in X:n\in \mathbb{N}, m\in \mathbb{Z}\}$ is dense in $X$. We let $\phi:\mathbb{N}\times \mathbb{Z} \to \mathbb{N}\times \mathbb{Z}$ be given by $\phi(n,m)=(n,m)$ if $z_{n,m}=0$ and by $\phi(n,m)=(k,j)$ if $z_{n,m}\ne 0$ so that $z_{n,m}=z_{k,j}$, that for any $k'<k$, any $j'\in \mathbb{Z}$, $z_{n,m}\ne z_{k',j'}$ and that for any $j'>j$, $z_{n,m}\ne z_{k,j'}$. The map $\phi$ is well-defined because each sequence $z_n\in c_0(X,\mathbb{Z})$. Then it suffices to consider $X_0=Y_0=\{z_{n,m}\in X:n\in \mathbb{N}, m\in \mathbb{Z}\}$, $S(0)=0$ and if $\phi(n,m)=(k,j)$ and $z_{n,m}\ne 0$ to let $Sz_{n,m}=z_{k,j+1}$.\\
$(2)\Rightarrow (1)$. It suffices to remark that for every $k\in \mathbb{Z}$, for every $x_k\in X_0$, we have
\[(\cdots,T^2x_k,T x_k,x_k,S x_{k},S^2x_k\cdots)\in c_0(X,\mathbb{Z}).\]
We can then conclude by applying Proposition~\ref{chaos-uni}.\\
$(2) \Rightarrow (3)$ Obvious.\\
$(3) \Rightarrow (5)$ Obvious.\\
$(4) \Rightarrow (5)$ Obvious.\\
$(5) \Rightarrow (6)$ Theorem~\ref{kit5}.\\
$(6) \Leftrightarrow (7)$ Corollary~\ref{bilmix}.\\
$(5) \Rightarrow (4)$ Since $(5) \Rightarrow (6)$, we know that $T$ is mixing. Therefore, it is enough to show that if $T$ is mixing, there are always maps $S_n:X\to X$ such that for each $y\in X$, $S^ny\to 0$  and $T^nS_ny\to y$. Let $y\in X$. Since $T$ is mixing, there exists an increasing sequence $(N_n)_{n\ge 1}$ such that for any $k\ge N_n$, $T^k(B_X(0,2^{-n}))\cap B_X(y,2^{-n})\ne \emptyset$. We can therefore let for any $k\le N_1$, $S_ky=0$ and for any $N_n\le k<N_{n+1}$, $S_ky=y_k$ with $y_k\in B_X(0,2^{-n})$ and $T^ky_k\in B(y,2^{-n})$ so that $S_ky\to 0$ and $T^kS_ky\to y$. Since $(5)$ gives us also a dense set $X_0$ in $X$ such that $T^nx\to 0$ for each $x\in X_0$, we get $(4)$.\\

To complete the proof, we still need to remark that there exists an operator satisfying $(6)$ but not $(5)$, that there exists an operator satisfying $(4)$ but not $(3)$ and that there exists an operator satisfying $(3)$ but not $(2)$.\\

An example of operator satisfying $(6)$ but not $(5)$ was given by Grivaux~\cite{Gri}. She showed that there exists a mixing operator $T$ on $X$ such that for every $x\in X\backslash\{0\}$, $T^nx$ does not tend to $0$. In particular, $T$ does not satisfy $(5)$.\\

An example of an operator satisfying $(4)$ but not $(3)$ is given by the weighted shift $T$ on $c_0(V)$ considered in the proof of Proposition~\ref{mixc0}. Indeed, we showed in this proof that for any $y\in X_0=Y_0=c_{00}(V)$, $T^n y$ tends to $0$ because $V$ is a rooted tree and that there exist maps $S_n:Y_0\to Y_0$ such that $T^nS_n y=y$ and $S_n y$ tends to $0$. In other words, $T$ satisfies the Hypercyclicity Criterion along $(n)$ with $X_0=Y_0$. However, $T$ does not satisfy the Kitai Criterion because we know that if $T$ satisfied this criterion then for every $y\in Y_0$, we would get a fixed point $(y,Sy,S^2y,S^3y,\dots)$ for $B_T$ on $c_0(X,\mathbb{N})$. It would then follow from Corollary~\ref{cor-chaos-uni} that $B_T$ is chaotic on $c_0(X,\mathbb{N})$ and we proved in Proposition~\ref{mixc0} that $B_T$ is not chaotic on $c_0(X,\mathbb{N})$.\\
 
 It will be more difficult to get an operator satisfying $(3)$ but not $(2)$. To this end, we are going to define an operator that can be seen as a "weighted shift on graph". Let $X=\ell^1(V)$ where $V$ consists of a family $(e_n)_{n\ge 1}$ and a family $(e^{(n,k)}_j)_{n,k,j\ge 1}$. We consider the operator $T$ given by
\[Te_{n+1}=2e_{n} \quad \text{for any $n\ge 1$ and } \quad Te_1=\sum_{n=1}^{\infty}2^{-n}e^{(n,1)}_1\]
and
\[Te^{(n,k)}_{j+1}=w_{n,k,j+1}e^{(n,k)}_{j} \quad \text{for any $j\ge 1$ and } \quad Te^{(n,k)}_1=2e^{(n,k+1)}_1\] 
where $w_{n,k,j}=4$ for any $k> n+1$ and any $j\ge 2$ and where for $1\le k\le n+1$
\[w_{n,k,j}= \left\{
    \begin{array}{ll}
        4 & \mbox{if } 2\le j\le n+k+1 \\
        \frac{1}{(n+1)4^{n+k}2^{n-k+1}} & \mbox{if } j=n+k+2\\
        2& \mbox{if } j>n+k+2
    \end{array}
\right. .\]

We remark that $T$ is well-defined and continuous on $\ell^1(V)$. We first show that $T$ satisfies the Kitai Criterion. To this end, we let $Y_0=c_{00}(V)$ and the linear map $S$ be given for any $n,k,j\ge 1$ by 
\[Se_n=\frac{1}{2}e_{n+1} \quad\text{ and } \quad Se^{(n,k)}_j=\frac{1}{w_{n,k,j+1}}e^{(n,k)}_{j+1}.\]
It is easy to check that for all $y\in Y_0$, we have $TSy=y$ and $S^ny\to 0$. In order to get the desired dense set $X_0$ such that $T^nx$ tends to $0$ for any $x\in X_0$, we show that $T$ has a dense generalized kernel so that we can consider $X_0=\bigcup_{m\ge 1}\ker(T^m)$. 

Let $n,k,j\ge 1$. For any $K>\max\{k,n+1\}$, the vector
\[y_K=e^{(n,k)}_{j}-2^{K-k}\frac{\prod_{i=2}^{j}w_{n,k,i}}{\prod_{i=2}^{j+K-k}w_{n,K,i}}e^{(n,K)}_{j+K-k}\]
belongs to the generalized kernel of $T$ since \[T^{j+K-k-1}y_K=\left(\prod_{i=2}^{j}w_{n,k,i}\right)T^{K-k}e^{(n,k)}_{1}-2^{K-k}\left(\prod_{i=2}^{j}w_{n,k,i}\right)e^{(n,K)}_{1}=0.\]
Moreover, we have that $(y_K)_K$ tends to $e^{(n,k)}_{j}$ as $K$ tends to $\infty$ since for $K>n+1$, $\prod_{i=2}^{j+K-k}w_{n,K,i}=4^{j+K-k-1}$.
In a similar way, for any $K\ge n$, the vector
\[z_K=e_n-2^{n+K-2} \sum_{m=1}^{\infty}\frac{2^{-m}}{\prod_{i=2}^{n+K}w_{m,K,i}}e^{(m,K)}_{n+K}\]
belongs to the generalized kernel since
\[T^{n+K-1}z_K=2^{n-1}\sum_{m=1}^{\infty}2^{-m}T^{K-1} e^{(m,1)}_1-2^{n+K-2} \sum_{m=1}^{\infty}2^{-m}e^{(m,K)}_{1}=0.\]
We remark that if $1\le m<K-1$ then $\prod_{i=2}^{n+K}w_{m,K,i}=4^{n+K-1}$ because $K>m+1$ and if $m\ge K-1$ then $\prod_{i=2}^{n+K}w_{m,K,i}=4^{n+K-1}$ because $n+K\le 2K\le m+K+1$.
Therefore, the sequence $(z_K)_K$ tends to $e_n$ as $K$ tends to $\infty$ and we deduce that $T$ has a dense generalized kernel and thus satisfies the Kitai Criterion.

We have still to prove that $T$ does not satisfy the Kitai Criterion with $X_0=Y_0$. We will actually show that for any increasing sequence $(n_l)$, $T$ does not admit a dense set $X_0$ in $X$ and a map $S:X_0\to X_0$ such that
\begin{enumerate}
\item $T^{n_l}x\to 0$ for each $x\in X_0$,
\item $S^{n_l}y\to 0$ for each $y\in X_0$,
\item $TSy=y$ for each $y\in X_0$.
\end{enumerate} Assume that these conditions are satisfied for an increasing sequence $(n_l)$ and a dense set $X_0$. Then, there exists $x\in X_0$ such that $\|x-e_1\|<\frac{1}{4}$ and $L\ge 1$ such that for every $l\ge L$, $\|T^{n_l}x\|<\frac{1}{2}$.
Let $l\ge L$ and $y=T^{n_l}x$. For any $n\ge 1$, we then have $|y^{(n,n_l)}_1|<\frac{1}{2}$ and
\[y^{(n,n_l)}_1=\frac{2^{n_l-1}}{2^n}x_1+\sum_{k=1}^{n_l}2^{n_l-k}\left(\prod_{i=2}^{k+1} w_{n,k,i}\right)x^{(n,k)}_{k+1}.\]
In particular, by considering $n=n_l-1$ and since  $|x_1|>\frac{3}{4}$, we deduce that 
\[\sum_{k=1}^{n_l}2^{n_l-k}\left(\prod_{i=2}^{k+1} w_{n_l-1,k,i}\right)|x^{(n_l-1,k)}_{k+1}|\ge 1/4.\]
Thus, there exists $1\le k_l\le n_l$ such that 
\[2^{n_l-k_l}\left(\prod_{i=2}^{k_l+1} w_{n_l-1,k_l,i}\right)|x^{(n_l-1,k_l)}_{k_l+1}|\ge \frac{1}{4n_l}.\]
However, since $TS=Id$ on $X_0$, it then follows that
\[\|S^{n_l}x\|\ge \frac{1}{\prod_{i=k_l+2}^{k_l+n_l+1} w_{n_l-1,k_l,i}}|x^{(n_l-1,k_l)}_{k_l+1}|\ge \frac{1}{\prod_{i=2}^{k_l+n_l+1} w_{n_l-1,k_l,i}}\frac{1}{4n_l 2^{n_l-k_l}}=\frac{1}{4}.\]
We conclude that $S^{n_l}x$ does not tend to $0$. Contradiction.
\end{proof}

This last counterexample allows us to answer an open question posed in \cite{GrPe11} and related to the Gethner-Shapiro Criterion (see \cite{GeSha}). We recall that this criterion is the following.

\begin{theorem}[Gethner-Shapiro Criterion along $(n_k)$]
If there are dense subsets $X_0,Y_0$ in $X$ and a map $S:Y_0\to Y_0$ such that
\begin{enumerate}
\item $T^{n_k}x\to 0$ for each $x\in X_0$,
\item $S^{n_k}y\to 0$ for each $y\in Y_0$,
\item $TSy=y$ for each $y\in Y_0$,
\end{enumerate}
then $T$ is weakly mixing.
\end{theorem}

It is known that the following assertions are equivalent (see \cite{BePe} and \cite{Pe}):
\begin{itemize}
\item there exists $(n_k)$ such that $T$ satisfies the Gethner-Shapiro Criterion along $(n_k)$,
\item there exists $(n_k)$ such that $T$ satisfies the Hypercyclicity Criterion along $(n_k)$ with $X_0=Y_0$,
\item there exists $(n_k)$ such that $T$ satisfies the Hypercyclicity Criterion along $(n_k)$,
\end{itemize}

However, it was not known if when an operator satisfies the Gethner-Shapiro Criterion along some sequence, this operator has to satisfy the Gethner-Shapiro Criterion along some sequence with the additional assumption that $X_0=Y_0$. Thanks to the proof of Theorem~\ref{kitai Thm}, we can now answer this question in the negative.

\begin{theorem}
There exist an infinite-dimensional separable Banach space $X$ and $T\in L(X)$ such that $T$ satisfies the Gethner-Shapiro Criterion along $(n)$ but for all increasing sequences $(m_k)$, $T$ does not satisfy the Gethner-Shapiro Criterion along $(m_k)$ with $X_0=Y_0$.
\end{theorem}

\subsection{Frequent Hypercyclicity Criterion}

The first version of the Frequent Hypercyclicity Criterion was given by Bayart and Grivaux in \cite{BaGr06} and was stated as follows.

\begin{theorem}[Frequent Hypercyclicity Criterion]
Let $X$ be a Banach space and $T\in L(X)$.
Suppose that there are a dense subset $X_0$ of $X$ and a map $S:X_0\to X_0$ such that for all $x\in X_0$, the following assertions hold:
\begin{enumerate}
\item $\sum_{n=1}^{\infty} \|T^nx\|<\infty$,
\item $\sum_{n=1}^{\infty} \|S^n x\|<\infty$,
\item $TS x=x$,
\end{enumerate}
then $T$ is frequently hypercyclic.
\end{theorem}

However, this criterion is too strong to characterize frequent hypercyclicty. In particular, the frequent hypercyclicity criterion also implies that $T$ is chaotic and mixing. Another criterion allowing to get frequent hypercyclicity was given in \cite[Theorem 3]{BMPP16}. This second criterion can be stated for any notion of $\mathcal{A}$-hypercyclicity and thus for the notion of frequent hypercyclicity by considering $\A=\underline{D}$ where $\underline{D}$ is the family of sets with positive lower density.

\begin{theorem}[$\underline{D}$-Hypercyclicity Criterion]\label{Ahypc}
Let $X$ be a separable Banach space and $T\in \mathcal{L}(X)$. If there exist a dense subset $Y_0\subset X$, $S_n:Y_0\rightarrow X$, $n\ge 0$, and disjoint sets $A_k\in \underline{D}$, $k\ge 1$, such that for each $y\in Y_0$,
\begin{enumerate}
\item $\sum_{n\in A_k}\|S_n y\|$ converges uniformly in $k\ge 1$,
\item for any $k_0\ge 1$, any $\varepsilon>0$, there exists $k\ge k_0$ such that
for any $n\in \bigcup_{l\ge 1} A_l$, we have
\[\sum_{i\in A_k\backslash\{n\}}\|T^nS_iy\|\le \varepsilon,\]
and such that for any $\delta>0$, there exists $l_0\ge 1$ such that for any $n\in \bigcup_{l\ge l_0} A_l$, we have
\[\sum_{i\in A_k\backslash\{n\}}\|T^nS_iy\|\le \delta,\]
\item $\sup_{n\in A_k}\|T^nS_ny-y\|\to 0$ as $k\to\infty$,
\end{enumerate}
then $T$ is frequently hypercyclic.
\end{theorem}

This criterion will have the advantage of characterizing when $B_T$ is frequently hypercyclic on $\ell^1(X,\mathbb{Z})$.

\begin{theorem}
Let $X$ be a separable Banach space and $T\in L(X)$.
Given the following assertions : 
\begin{enumerate}
\item $B_T$ is chaotic on $\ell^1(X,\mathbb{Z})$,
\item $T$ satisfies the Frequent Hypercyclicity Criterion,
\item $T$ satisfies the $\underline{D}$-Hypercyclicity Criterion,
\item $B_T$ is frequently hypercyclic on $\ell^1(X,\mathbb{Z})$,
\item $T$ is frequently hypercyclic on $X$,
\end{enumerate}
then 
\[(1)\Leftrightarrow (2) \Rightarrow (3) \Leftrightarrow (4) \Rightarrow (5).\]
and the other implications are false in general.
\end{theorem}
\begin{proof}
We proceed as in the proof of Theorem~\ref{kitai Thm} to show that $(1)\Leftrightarrow (2)$.\\
$(1)\Rightarrow (2)$. Let $\text{Fix}(B_T)$ be the set of fixed points for $B_T$. We know by Proposition~\ref{chaos-uni} that there exists a sequence $(z_n)_{n\in \mathbb{N}}\subset \text{Fix}(B_T)\backslash\{0\}$ such that $\{z_{n,m}\in X:n\in \mathbb{N}, m\in \mathbb{Z}\}$ is dense in $X$. We let $\phi:\mathbb{N}\times \mathbb{Z} \to \mathbb{N}\times \mathbb{Z}$ be given by $\phi(n,m)=(n,m)$ if $z_{n,m}=0$ and by $\phi(n,m)=(k,j)$ if $z_{n,m}\ne 0$ so that $z_{n,m}=z_{k,j}$, that for any $k'<k$, any $j'\in \mathbb{Z}$, $z_{n,m}\ne z_{k',j'}$ and that for any $j'>j$, $z_{n,m}\ne z_{k,j'}$. The map $\phi$ is well-defined because each sequence $z_n\in \ell^1(X,\mathbb{Z})$. We may conclude by considering $X_0=Y_0=\{z_{n,m}\in X:n\in \mathbb{N}, m\in \mathbb{Z}\}$, $S(0)=0$ and $Sz_{n,m}=z_{k,j+1}$ when $z_{n,m}\ne 0$ and  $\phi(n,m)=(k,j)$.\\
$(2)\Rightarrow (1)$. It suffices to remark that for every $k\in \mathbb{Z}$, for every $x_k\in X_0$, we have
\[(\cdots,T^2x_k,T x_k,x_k,S x_{k},S^2x_k\cdots)\in \ell^1(X,\mathbb{Z}).\]
We can then conclude by applying Proposition~\ref{chaos-uni}.\\
$(1) \Rightarrow (4)$ Proposition~\ref{chaosfhc}.\\
$(3)\Rightarrow (4)$ It suffices to remark that if $T$ satisfies the $\underline{D}$-Hypercyclicity Criterion for a dense set $Y_0$ in $X$, a sequence of maps $S_n:Y_0\to X$ and a sequence of sets $(A_k)$ of positive lower density then $B_T$ also satisfies the $\underline{D}$-Hypercyclicity Criterion with $Y'_0=c_{00}(Y_0,\mathbb{Z})$, $S'_n(y)=z$ where $z_{k+n}=S_n(y_k)$ for all $k\in \mathbb{Z}$ and the same family $(A_k)$.\\
$(3)\Rightarrow (5)$ Theorem~\ref{Ahypc}.\\
$(4)\Rightarrow (3)$ Let $x$ be a frequently hypercyclic vector for $B_T$, let $(C_k)_{k\ge 1}$ be an increasing sequence tending to infinity and $(\varepsilon_k)_{k\ge 1}$ a decreasing sequence tending to $0$. We consider $B_k=N_{B_T}(x,B(C_kx,\varepsilon_k))\backslash\{0\}$. Each set $B_k$ is thus a set of positive lower density. We also consider an increasing sequence of integers $(N_l)_{l\ge 1}$ such that
\begin{equation}
C_l\sum_{i\notin[-N_l,N_l]}\|x_i\|_X\le \frac{1}{l}.
\label{paramC}
\end{equation}
Thanks to \cite[Lemma 2.2]{MaMePu}, we know that we can find a family $(A_k)$ such that $A_k\subset B_k$, $A_k$ has positive lower density and for any $n\in A_k$, any  $m\in A_j$, if $n\ne m$ then 
\begin{equation}
|n-m|> N_k+N_j.
\label{sep}
\end{equation}

Let $y\in X$. We let $z\in \ell^1(X,\mathbb{Z})$ be given by $z_0=y$ and $z_k=0$ for all $k\ne 0$. Since $x$ is hypercyclic for $B_T$, there exists an increasing sequence $(M_j)_{j\ge 1}$ with $M_1=0$ such that $\|B^{M_j}_Tx-z\|\le \frac{\|x-z\|}{j}$ and a non-decreasing sequence $(j_k)_{k\ge 1}$ tending to $\infty$ such that
\[M_{j_k}\le N_k \quad \text{and}\quad \frac{\|B^{M_{j_k}}_T\|}{C_k}\to 0.\]
 We then let for any $n\in A_k$
\[S_n y=\frac{1}{C_k}T^{M_{j_k}} x_{M_{j_k}+n} \]
so that
\begin{align*}
\|T^nS_ny-y\|_X&=\|\frac{1}{C_k}T^{M_{j_k}+n}x_{M_{j_k}+n}-y\|_X\\
&\le \|\frac{1}{C_k}B^{M_{j_k}+n}_Tx-z\|\\
&\le 
\|\frac{1}{C_k}B^{M_{j_k}+n}_Tx-B^{M_{j_k}}_Tx\|+\|B^{M_{j_k}}_Tx-z\|\\
&\le 
\frac{\|B_T^{M_{j_k}}\|}{C_k} \|B^{n}_Tx-C_k x\|+ \frac{\|x-z\|}{j_k}\\
&\le \frac{\|B^{M_{j_k}}_T\|\varepsilon_k}{C_k}+ \frac{\|x-z\|}{j_k}
\to 0.
\end{align*}
Moreover, we have 
\[\sum_{n\in A_k}\|S_ny\|_X= \frac{1}{C_k} \sum_{n\in A_k} \|T^{M_{j_k}} x_{M_{j_k}+n}\|_X
\le \frac{1}{C_k}\|B^{M_{j_k}}_Tx\|\le \frac{\|B^{M_{j_k}}_T\|}{C_k}\|x\|.\]
Since $\frac{\|B^{M_{j_k}}_T\|}{C_k}$ tends to 0, we deduce that $\sum_{n\in A_k}\|S_ny\|_X$ converges uniformly in $k\ge 1$. Finally, if $n\in A_l$, we have
\begin{align*}
\sum_{i\in A_k\backslash\{n\}}\|T^nS_iy\|_X&=\sum_{i\in A_k\backslash\{n\}}\frac{1}{C_k}\|T^{M_{j_k}+n} x_{M_{j_k}+i}\|_X\\
&\le \frac{\|B_T^{M_{j_k}}\|}{C_k}\sum_{i\in A_k\backslash\{n\}}\|T^nx_{M_{j_k}+i}\|_X\\
&\le  \frac{\|B_T^{M_{j_k}}\|}{C_k}\sum_{i\notin[-N_l,N_l]}\|T^nx_{n+i}\|_X \quad \text{(by \eqref{sep} and $M_{j_k}\le N_k$)}\\
&\le  \frac{\|B_T^{M_{j_k}}\|}{C_k}\left(\sum_{i\notin[-N_l,N_l]}\|T^{n}x_{n+i}-C_lx_i\|_X
+ C_l\sum_{i\notin[-N_l,N_l]}\|x_i\|_X\right)\\
&\le  \frac{\|B_T^{M_{j_k}}\|}{C_k}\left(\|B_T^{n}x-C_lx\|
+ \frac{1}{l}\right) \quad \text{(by \eqref{paramC})}\\
&\le  \frac{\|B_T^{M_{j_k}}\|}{C_k}\left(\varepsilon_l+ \frac{1}{l} \right).
\end{align*}
Since $\frac{\|B_T^{M_{j_k}}\|}{C_k}$ and $\left(\varepsilon_l+ \frac{1}{l} \right)$ tend to $0$, we get the desired inequalities.\\

To conclude, we know that $(4)$ does not imply $(1)$ thanks to Proposition~\ref{fhcnotchaos} and that
(5) does not imply (4) by considering a frequently hypercyclic operator $T$ that is not mixing so that $B_T$ is not mixing on $\ell^1(X,\mathbb{Z})$ by Corollary~\ref{bilmix} and thus not frequently hypercyclic on $\ell^1(X,\mathbb{Z})$ by Corollary~\ref{fhcmix}.
\end{proof}

\end{document}